\documentclass[11pt]{amsart}
\usepackage{amsmath,amssymb,times,color}
\usepackage{graphicx}
\usepackage{fig4tex}
\usepackage{dsfont}
\definecolor{webred}{rgb}{0.75,0,0}
\definecolor{webgreen}{rgb}{0,0.75,0}
\usepackage[citecolor=webgreen,colorlinks=true,linkcolor=webred]{hyperref}
\usepackage{stmaryrd} 
\usepackage{mathrsfs}

\newtheorem{thm}{Theorem}[section]

\newtheorem{lem}[thm]{Lemma}
\newtheorem{prop}[thm]{Proposition}

\newtheorem{hyp}[thm]{Hypothesis}
\newtheorem{notn}[thm]{Notation}

\theoremstyle{definition}

\theoremstyle{remark}
\newtheorem{rem}[thm]{Remark}

\oddsidemargin=\evensidemargin

\numberwithin{equation}{section}

\newcommand{\Div}{\operatorname{\mathrm{div}}}

\newcommand{\rot}{\operatorname{\mathrm{curl}}}

\newcommand{\db}[1]{_{\raise-0.3ex\hbox{$\scriptstyle #1$}}}
\newcommand{\dd}[1]{_{\raise-1.5pt\hbox{$\scriptstyle #1$}}}

\newcommand{\di}{\displaystyle}

\newcommand{\dr}{{\rm d}}

\newcommand  {\N}{{\mathbb N}}

\newcommand  {\R}{{\mathbb R}}

\newcommand {\Id}{\mathbb {I}}

\renewcommand  {\H}{{\mathrm H}}
\newcommand  {\PH}{{\mathrm {PH}}}

\renewcommand  {\L}{{\mathrm L}}

\newcommand  {\eps}{\varepsilon}

\newcommand  {\nn}{\mathsf n}

\newcommand  {\xx}{\boldsymbol{\mathsf x}}

\newcommand  {\bH}{\mathbf{H}}
\newcommand  {\bL}{\mathbf{L}}

\newcommand {\bd}{\mathbf{d}}

\renewcommand {\sl}{\mathsf{l}}

\newcommand {\su}{\mathsf{u}}

\newcommand {\sq}{\mathsf{q}}
\newcommand {\sr}{\mathsf{r}}
\renewcommand {\sf}{\mathsf{f}}
\newcommand {\sg}{\mathsf{g}}
\newcommand {\sz}{\mathsf{z}}
\newcommand {\sh}{\mathsf{h}}

\newcommand  {\sw}{\mathsf{w}}

\newcommand {\ds}{\longrightarrow}

\def\dsp{\displaystyle}

\newcommand {\dzz} {  \partial_{\sz \sz}}

\newcommand {\dz} {  \partial_{\sz }}

\def \dd#1#2{\frac{\partial #1}{\partial #2}}

\newcommand {\sv}{v}
\renewcommand {\sp}{p}

\newcommand{\vepsp}{{v}^+_{\eps}}
\newcommand{\vepsm}{{v}^-_{\eps}}

\newcommand{\Vplus}{\mathrm{v}^+}
\newcommand{\Vmb}{{\mathrm{v}}^-}
\newcommand{\Vmt}{\widetilde{\mathrm{v}}^-}
\newcommand{\pmob}{{\mathrm{p}}^-}
\newcommand{\Pplus}{{\mathrm{p}}^+}
\newcommand  {\fp}{\widetilde{\mathrm{p}}^-}
\newcommand{\bk}{\mathrm{k}}

\newcommand{\sk}{\mathrm{k}}
\newcommand{\bsigma} {\boldsymbol{\sigma}}

\newcommand  {\fv}{\Vmt}

\definecolor{mpurple}{rgb}{0.6,0,0.8}
\definecolor{myblue}{rgb}{0.,0.2,0.8}
\definecolor{mygreen}{rgb}{0,0.75,0.0}
\definecolor{mred}{rgb}{0.9,0,0}
\definecolor{mbrun}{rgb}{0.8,0.5,0}

\begin{document}


\title[Asymptotic Study for Stokes-Brinkman  model]{Asymptotic Study for Stokes-Brinkman  model with jump embedded transmission conditions}

\author{Philippe Angot\textsuperscript{1}, Gilles Carbou\textsuperscript{2}, Victor P\'eron\textsuperscript{2,3} }

\thanks{\noindent \textsuperscript{1} Laboratoire d'Analyse, Topologie et Probabilit\'es (LATP), UMR CNRS 7353 Equipe d'Analyse Appliqu\'ee
Centre de Math\'ematiques et Informatique (CMI),
Aix-Marseille Universit\'e, Technopole Chateau-Gombert,
39, rue F. Joliot Curie,
13453 Marseille Cedex 13
FRANCE}

\thanks{\noindent \textsuperscript{2} Laboratoire de Math\'ematiques et de leurs Applications de Pau, UMR CNRS 5142, 
B\^atiment IPRA, Universit\'e de Pau et des Pays de l'Adour, 
Avenue de l'Universit\'e - BP 1155,
64013 PAU CEDEX}

\thanks{\noindent \textsuperscript{3} Magique 3d, INRIA Bordeaux Sud-Ouest, 
 Universit\'e de Pau et des Pays de l'Adour, 
Avenue de l'Universit\'e - BP 1155, 
64013 PAU CEDEX}

\begin{abstract}
In this paper, one considers the coupling of a Brinkman model and Stokes equations with jump embedded transmission conditions. In this model, one assumes that the viscosity  in the porous region is very small. Then we derive a Wentzel--Kramers--Brillouin (WKB) expansion in power series of the square root of this small parameter for the velocity and the pressure which are solution of the transmission problem. This WKB expansion is justified rigorously by proving uniform errors estimates.  
 
\end{abstract}
\date{\today, Version 6}

\maketitle

\tableofcontents

\section{Introduction}

We address the problem of fluid flow modeling in complex media which combine porous regions and fluid regions with free flow. This issue holds for instance in the study of aquifer media made up of a porous media containing craks and conduits (see \cite{CGW10}) and also the passive control of the flow around an obstacle covered by a porous thin layer (see \cite{BM01}).

In this paper the free flow satisfies the linear Stokes equation. In the porous media, we consider two models, the Brinkman  and the Darcy models. There are several interface conditions in the literature. In the case of the Stokes-Brinkman  coupling, the simpler interface condition is the continuity of the velocity and the normal stress. More accurate models are given by Ochoa-Tapia \& Whitaker transmission conditions, by Beavers \& Joseph conditions or Beaver, Joseph \& Saffman conditions (see \cite{CGW10}). In this paper, we deal with the more general jump embedded transmission conditions described in \cite{An11}. This condition links the jumps and the averages of both the velocity and the shear stress on the interface. A small parameter appears in this system, in particular the equivalent viscosity in the porous part is small, so that when this parameter tends to zero, we expect that the flow in the porous part will be described by the Darcy law. Our interest lies in the asymptotic study justifying the obtention of the limit model. In particular we describe the boundary layer due to the jump conditions appearing in the porous medium.

\vspace{2mm}

Let us describe the Stokes-Brinkman model with jump embedded transmission conditions.   The problem is set in the domain $\Omega\subset \R^3$ made of a fluid  region $\Omega_{+}$ and a porous subdomain $\Omega_{-}$. 
\begin{figure}[h]
\begin{center}
\figinit{0.8pt}
\figpt 1:(-100,20)
\figpt 2:(-30,70)\figpt 3:(50,50)
\figpt 4:(80,0)\figpt 5:(30,-40)
\figpt 6:(-30,0)\figpt 7:(-80,-20)
\figpt 8:(-50,20)\figpt 9:(-20,50)\figpt 10:(50,20)
\figpt 12:(10,10)\figpt 13:(0,0)
\figpt 14:(-20,5)%
\figpt 16:(-70,20)\figpt 15:(-10,30)
\figpt 17:(40,10) \figpt 18:(-60,-20)
\figpt 21:(50,65)
\figpt 19:(61,65) \figpt 20:(17,-7)
\psbeginfig{}
\pscurve[1,2,3,4,5,6,7,1,2,3]
\pssetfillmode{yes}\pssetgray{0.8}
\pscurve[8,9,10,12,13,14,8,9,10]
\pssetfillmode{no}\pssetgray{0}
\psarrow[12,20]
\psendfig

\figvisu{\figBoxA}{{Figure 1}\ --\ The domain $\Omega$ and the subdomains $\Omega_{-}$, $\Omega_{+}$}{
\figwritew 15: $\Omega_{-}$(6pt)
\figwritec [16]{$\Omega_{+}$}
\figwritec [17]{$\Sigma$}
\figwritec [18]{$\Gamma$}
\figwritee 20: $\nn$(1pt)
  }
\centerline{\box\figBoxA}
 \label{F1}
\end{center}
\end{figure}
We assume that the domains $\Omega_{-}$ and $\Omega_{+}$ are Lipschitz and bounded, and that $\Omega_- \subset \overline{\Omega_-} \subset \Omega$. We denote $\Sigma=\partial \Omega_-$ so that se have $\Omega=\Omega_-\cup \Sigma\cup \Omega_+$ (see Figure \ref{F1}). We denote by $\nn$ the outward unit normal at $\partial \Omega_-$. 

On the fluid region $\Omega_{+}$, the velocity $\sv_{\eps}^+$ and the pressure $\sp_{\eps}^+$ satisfy the Stokes equations. In the porous region $\Omega_{-}$,  the velocity $\sv_{\eps}^-$ and the pressure $\sp_{\eps}^-$ satisfy a Brinkman model. We couple these models by the Beaver-Joseph conditions at the common boundary $\Sigma$. These conditions link the jumps of the velocity and the normal stress vector with the averages of these quantities on $\Sigma$.

We denote by  $ \bsigma^+(\sv_{\eps}^+,\sp_{\eps}^+ )$ (resp. $\bsigma^-(\sv_{\eps}^-,\sp_{\eps}^- )$) the stress tensor in the fluid (resp. in the porous) medium :
$$
\dsp \bsigma^+(\sv_{\eps}^+,\sp_{\eps}^+ )= 2\mu \bd(\sv_{\eps}^+)- \sp_{\eps}^+ \Id , \quad
\dsp  \bsigma^-(\sv_{\eps}^-,\sp_{\eps}^- )=2\eps \bd(\sv_{\eps}^-)- \sp_{\eps}^- \Id,$$
with
\begin{itemize}
\item  $ \dsp 
\bd(v) =\frac 1 2 \left(\nabla v + \nabla^\perp v\right)$, that is $\dsp  (\bd(v))_{ij}=\frac 1 2 \left( \dd{v^j}{x_i} + \dd{v^i}{x_j}\right)$,
\item  $\mu$ is the viscosity of the fluid and $\eps$ is the effective viscosity in the porous medium.
\end{itemize}

\vspace{2mm}
The problem writes 
\begin{equation}
\label{0SBBJ}
 \left\{
   \begin{array}{lll}
-\nabla \cdot \bsigma^-(\sv_{\eps}^-,\sp_{\eps}^- ) + \kappa\sv_{\eps}^{-}=\sg^-
 \quad&\mbox{in}\quad \Omega_{-}
\\[0.5ex]
-\nabla \cdot  \bsigma^+(\sv_{\eps}^+,\sp_{\eps}^+ )=\sg^+
 \quad&\mbox{in}\quad \Omega_{+}
\\[0.5ex]
\nabla\cdot \sv_{\eps}^-=0 \quad&\mbox{in}\quad\Omega_{-}
\\[0.5ex]
\nabla\cdot \sv_{\eps}^+=0 \quad&\mbox{in}\quad \Omega_{+}
\\[0.5ex]
{\bsigma}^+(\sv_{\eps}^+,\sp_{\eps}^+ )\cdot \nn = \bsigma^-(\sv_{\eps}^-,\sp_{\eps}^- )\cdot \nn +\boldsymbol{M}\{\sv_{\eps}\} +\sl  &\mbox{on} \quad \Sigma
\\[0.5ex]
\{ \bsigma(\sv_\eps,\sp_\eps)\cdot \nn  \}=\boldsymbol{S}(\sv_{\eps}^+ -\vepsm) + \sh &\mbox{on} \quad \Sigma
\\[0.2ex]
\sv_{\eps}=0
  \quad &\mbox{on}\quad \Gamma \
   \end{array}
    \right.
\end{equation}
where:

\begin{itemize}
\item  $\kappa>0$ is fixed positive constant.
\item In the previous equations set on $\Sigma$, we denote by $\{\sw\}$ the mean value across $\Sigma$ of $\sw$.
\item The right-hand sides in \eqref{0SBBJ} are given data which are  defined as follow : $\sg^\pm: \Omega_{\pm}\rightarrow \R^3$ are  vector fields in $\left(\L^2(\Omega_{\pm})\right)^3$, $\sh\in\left(\H^{-1/2}(\Sigma)\right)^3$,  and $\sl\in\left(\H^{-1/2}(\Sigma)\right)^3 $  are vector fields defined on $\Sigma$.
\item The matrix $\boldsymbol{M}$ is zero on $\nn^\perp$ and satisfies: $$\boldsymbol{M}\xi = \beta (\xi \cdot \nn)\nn$$with $\beta >0$. 
\item The matrix $\boldsymbol{S}$ satisfies $\boldsymbol{S}_{\vert  \nn^\perp}=\alpha \Id_{\vert\nn^\perp}$ and $\dsp \boldsymbol{S}(\nn)=\frac{1}{\eps}\nn$:
$$\boldsymbol{S}(\xi ) = \frac{1}{\eps}(\nn\cdot \xi)\nn + \alpha (\xi - (\nn\cdot \xi)\nn),$$
with $\alpha>0$.
\end{itemize}

We remark that using the divergence free conditions, we can replace the first equation in \eqref{0SBBJ} by
\begin{equation}
\label{putain}
- \eps \Delta \sv_{\eps}^{-} + \nabla \sp_{\eps}^-+\kappa\sv_{\eps}^{-}=\sg^-
 \quad\mbox{in}\quad \Omega_{-}
\end{equation}

\begin{notn}

For any set $\mathcal{O}\subset\R^3$ we denote $\bL^2(\mathcal{O})$ the space $\left(\L^2(\mathcal{O})\right)^3$ and 
$\bH^s(\mathcal{O})$ the space $\left(\H^{s}(\mathcal{O})\right)^3$. 

\end{notn}

\section{Statement of the Main Results}

In order to prove existence of weak solutions for \eqref{0SBBJ}, let us describe the associated  variational formulation.

We denote by $\PH^1 (\Omega)$  the space of vector fields $\sv\in\bL^2(\Omega)$ such that $\sv^+:= \sv_{\vert\Omega_{+}}\in\bH^1(\Omega_{+})$ and $\sv^-:= \sv_{\vert\Omega_{-}}\in\bH^1(\Omega_{-})$. We introduce a weak formulation of the problem \eqref{0SBBJ} for  $\sv=(\sv^+,\sv^-)$ in the space
$$
V=\{ \su \in \PH^1 (\Omega) \ | \, \Div \su^\pm =0  \quad  \mbox{in}\quad \Omega_{\pm}  , \quad \su^+ =0 \quad  \mbox{on}  \quad\Gamma
\} \ ,
$$ 
endowed with the piecewise $\H^1$ norm. Such a  variational formulation writes : Find $\sv=(\sv^-,\sv^+)  \in V$ such that
\begin{equation}
   \forall \su \in  V, \quad 
   a_{\eps}(\sv,\su) = b(\su) ,
\label{VP}
\end{equation}
where
\begin{multline*}
  a_{\eps} (\sv,\su) :=
 2 \eps \int_{\Omega_{-}} \bd( \sv^-) : \bd( \su^-)\, \dr\xx 
  + \int_{\Omega_{-}}\kappa\sv^- \cdot \su^-\, \dr\xx 
  +2 \mu  \int_{\Omega_{+}} \bd (\sv^+) : \bd( \su^+) \,\dr\xx
  \\
   +
   < \beta \{\sv \cdot \nn\}\, , \{\su \cdot \nn \}>_{-1/2, \Sigma}
   +    <\alpha [\sv_{\tau} ] \, , 
     [\su_{\tau} ]  >_{-1/2, \Sigma}
+\eps^{-1}
<[\sv  \cdot \nn ]   \, ,  
 [\su  \cdot \nn ]  >_{-1/2, \Sigma} 
  \ ,
\end{multline*}
with $< \, ,   >_{-1/2, \Sigma} $ being a duality pairing between $\H^{-1/2}(\Sigma)^3$ and    $\H^{1/2}(\Sigma)^3$, $ [\sv ] = \sv^+ -\sv^-$ denoting the jump of $\sv$ across $\Sigma$, and
$$
 b(\su)
     =   \int_{\Omega}  \sg \cdot  \su  \,\dr\xx 
     + 
   < \sh \, ,       [\su ]  >_{-1/2, \Sigma}
+
   < \sl \, ,\{\su\}>_{-1/2, \Sigma}
 \   ,
$$  
compare with \cite[Th 1.1]{An10}, \cite[Th 2.1]{An11}.  The duality pairing $< \, ,   >_{-1/2, \Sigma} $  between $\H^{-1/2}(\Sigma)^3$ and    $\H^{1/2}(\Sigma)^3$ coincides with the duality pairing in $\L^2(\Sigma)^3$. 
 
\vspace{2mm}
For $\eps>0$, applying Lax Milgram Theorem,  we prove the existence and uniqueness of weak solution for \eqref{0SBBJ} and we obtain uniform estimates of the solutions with respect to the parameter $\eps$. 

\begin{thm}
\label{thm32}
Let $\sg^\pm\in\bL^2(\Omega_{\pm})^3$ and $\sh,\sl \in\L^2(\Sigma)^3$. Then, the problem  \eqref{VP} has a unique solution $\sv=\sv_{\eps}\in V$ for all $\eps>0$. Moreover, for $\eps>0$ small enough, the following uniform estimate holds : 
\begin{multline}
\label{E5}
 \eps \|\bd( \sv^-) \|_{0,\Omega_{-}}^2
  +  \frac {\kappa}{4} \| \sv^- \|^2_{0,\Omega_{-}}  
  + \frac{ \mu C^2}{3} \|\sv^+ \|_{1,\Omega_{+}}^2
+  \frac{1}{4\eps} \| [\sv  \cdot \nn ] \|^2_{0,\Sigma}
+  \frac{\alpha}{4} \| [\sv ] \|^2_{0,\Sigma}
+\frac{\beta}{4}
\| \{\sv \cdot \nn\}   \|^2_{0,\Sigma}
\\
\leqslant
c
\left(  \|\sg\|^2_{0,\Omega}
+  \|\sh \|^2_{0,\Sigma} 
+ \|\sl\|^2_{0,\Sigma} 
\right)
  \ ,
\end{multline} 
with a constant $c=c(\mu, \kappa ,\alpha, \beta)$.
\end{thm}
\begin{rem}
This result still holds when the data $\sh$ and $\sl$ belong to the space $\H^{-\frac12}(\Sigma)^3$. For the sake of simplicity, we prove this lemma when $\sh$ and $\sl$ belong to the space $\L^2(\Sigma)^3$. 
 We eventually compare this stability result with \cite{An10,An11} where the author prove also energy estimates.  In \cite[Th 1.1]{An10} and \cite[Th 2.1]{An11}, estimates are non-necessary uniform, whereas estimates \eqref{E5}  are uniform with respect to the parameter $\eps$. 
\end{rem}

The asymptotic limit of the Stokes-Brinkman model towards the Stokes-Darcy one with Beavers-Joseph interface conditions is studied in \cite{An11,CGW10} in the case of a flat interface when the viscosity $\tilde \mu $ in the porous region is very small $\tilde \mu =\eps\ll 1$ and when the jump of the normal velocities is penalized. We address in this paper  the problem of the convergence of the model \eqref{0SBBJ} when the parameter $\eps $ tends to zero. 

When $\eps$ tends to zero, we formally converge to the following Stokes Darcy problem with Beavers \& Joseph interface conditions
\begin{equation}
\label{SD-BJC}
 \left\{
   \begin{array}{lll}
\nabla \sp_{0}^-+\kappa\sv_{0}^{-}=\sg^-
 \quad&\mbox{in}\quad \Omega_{-}
\\[0.5ex]
 -\nabla \cdot \bsigma(\sv_0^+,\sp_0^+)=\sg^+ \quad&\mbox{in}\quad \Omega_{+}
\\[0.5ex]
\nabla\cdot \sv_{0}^-=0 \quad&\mbox{in}\quad\Omega_{-}
\\[0.5ex]
\nabla\cdot \sv_{0}^+=0 \quad&\mbox{in}\quad \Omega_{+}
\\[0.5ex]  
\dsp \bsigma(\sv_0^+,\sp_0^+)\cdot \nn =   -\sp_{0}^- \nn+\frac{\beta}{2}\left(( \sv_{0}^++\sv_{0}^-)\cdot \nn\right)\nn   +\sl
 \quad &\mbox{on} \quad \Sigma
   \\[0.5ex] 
  (\sv_{0}^+ -\sv_{0}^-)\cdot \nn =  0\ 
  \quad &\mbox{on}\quad \Sigma    
     \\[0.5ex] 
  \sv^+_{0} =  0
  \quad &\mbox{on} \quad \Gamma    \ .
           \end{array}
    \right.
\end{equation}
This problem is incompatible with the limit (as $\eps$ tends to zero) of the tangential part of the interface condition $\dsp \{ \bsigma(\sv_\eps,\sp_\eps)\cdot \nn  \}=\boldsymbol{S}(\sv_{\eps}^+ -\vepsm) + \sh $ on $ \Sigma$ so that it appears a boundary layer inside the porous medium.  We describe this boundary layer by an asymptotic expansion at any order with a WKB method.   We derive this expansion in power series of the small parameter $\sqrt\eps$ for both the velocity $\sv_{\eps}$ and the pressure $\sp_{\eps}$ which are solutions of the transmission problem. This expansion is justified rigorously by proving uniform estimates for remainders of this expansion, Theorem \ref{DAS-WKB}. 

An immediate corollary of our asymptotic expansion is the following convergence theorem:

\begin{thm}
\label{thm-conv}
We assume that  the data in \eqref{0SBBJ} satisfy:
$$
\sg^-\in \bH^5(\Omega_{-}),\quad   \sg^+\in \bH^{4}(\Omega_{+}),\quad  \sl\in \bH^{\frac{9}{2}}(\Sigma)\quad\mbox{and}\quad\sh\in \bH^{\frac{9}{2}}(\Sigma).
$$
Then, the solution $\sv_{\eps}$ for \eqref{0SBBJ}  given by Theorem \ref{thm32} satisfies:
$$\begin{array}{ll}
\sv_\eps^+(x) = \sv_0^+(x)  + \sr_\eps^+(x)\quad&\mbox{ for } x\in\Omega_{+}\\
\\
\sv_\eps^-(x) = \sv_0^-(x) +\Vmt_0(x, \frac{d(x)}{\sqrt \eps}) + \sr_\eps^-(x)\quad&\mbox{ for } x\in\Omega_{-}
\end{array}$$
where 
\begin{itemize}
\item $\sv_0$ is the solution of \eqref{SD-BJC}, 
\item $\Vmt_0$ is a boundary layer term of the form $\dsp \Vmt_0(x,\sz)=\sw_0(x) \exp (-\sqrt \kappa \sz)$,
\item $d(x)$ is the euclidean distance to $\Sigma$,
\item $\sr_\eps$  is a  remainder terms.
\end{itemize}
This remainder term satisfies the following estimate:
$$
 \eps \|\bd(\sr_\eps^{-}) \|_{0,\Omega_{-}}^2
  +  \frac {\kappa}{4} \|\sr_\eps^{-} \|^2_{0,\Omega_{-}}  
  + \frac{ \mu C^2}{3} \|{\sr}_\eps^{+} \|_{1,\Omega_{+}}^2
\leqslant
C\eps^{\frac{1}{2}}.$$
\end{thm}

 The concept of WKB expansion is rather classical in the modeling of problems arising in fluid mechanics. For instance in \cite{CF03,C04,C08} the authors derive WKB expansions with boundary layer terms or thin layer asymptotics to describe penalization methods in the context of viscous incompressible flow.

In this work, one difficulty to validate the WKB expansion lies in the proof  of both existence and  regularity results for one part of the asymptotics  which appear  in this expansion at any order and which solve Darcy-Stokes problems with non-standard transmission conditions. It is  possible to tackle these problems by carefully introducing a Dirichlet--to--Neumann operator which lead us to prove well-posedness results  and elliptic regularity results simply for the Stokes operator with mixed boundary conditions.

The outline of the paper proceeds as follows. We prove the well-posedness result for Problem \eqref{VP} together with uniform estimates  with respect to the  small parameter is  Section \ref{UE}.  In Sections \ref{AE} and \ref{SDERA}, one exhibits a formal WKB expansion  for the solution of the transmission problem.  The equations satisfied by the asymptotics at any order are explicited in section \ref{SDERA} and existence and regularity results concerning the asymptotics which satisfy Darcy-Stokes problems with non-standard transmission conditions are claimed in Prop. \ref{profilbarre}. The proof of this proposition is postponed to section \ref{secregprof}. In section \ref{SERem}, one proves uniform errors estimates to validate this WKB expansion.

\section{Uniform estimates}
\label{UE}

In this section, we prove the well-posedness result for Problem \eqref{VP} together with uniform estimates, Theorem \ref{thm32}. 

\begin{notn}
We denote by $\| \cdot \|_{s,\mathcal{O}}$ the norm in the Sobolev space $\H^s(\mathcal{O})$. 
\end{notn}

{\em Proof of Theorem \ref{thm32}.}
Since $\Omega_+$ is a lipschitz bounded domain, since $\sv^+=0$ on $\partial \Omega$ for $\sv \in V$, we have the following  Poincar\'e inequality in $\Omega_+$:
\begin{equation}
\exists C>0 \quad \mbox{s.t.}\quad   \forall \sv \in V \quad \| \bd( \sv^+) \|_{0,\Omega_{+}} \geqslant C \| \sv^+\|_{1,\Omega_{+}} \ ,
\end{equation}
then, there hold $\forall \sv \in  V$, 
\begin{multline*}
  a_{\eps} (\sv,\sv)\geqslant
 2 \eps \|\bd(\sv^-) \|_{0,\Omega_{-}}^2
  + \kappa\| \sv^- \|^2_{0,\Omega_{-}}  
  + 2 \mu C^2 \|\sv^+ \|_{1,\Omega_{+}}^2
+\eps^{-1} \| [\sv  \cdot \nn ] \|^2_{0,\Sigma}
\\
+\beta 
\| \{\sv \cdot \nn\}   \|^2_{0,\Sigma}
+  \alpha\| \sv_{\tau}  \|^2_{0,\Sigma}
 \ .
\end{multline*}
Hence, $ a_{\eps}$ is $V$-coercive, and according to the Lax-Milgram Lemma, the problem  \eqref{VP} has a unique solution $\sv_{\eps}\in V$ for all $\eps>0$.  We also infer : if $\sv$ satisfies \eqref{VP}, then $\forall \sv \in  V$, 
\begin{multline*}
 2\eps \|\bd(\sv^-) \|_{0,\Omega_{-}}^2
  + \kappa\| \sv^- \|^2_{0,\Omega_{-}}  
  +  2\mu C^2 \|\sv^+ \|_{1,\Omega_{+}}^2
+\eps^{-1} \| [\sv  \cdot \nn ] \|^2_{0,\Sigma}
+\beta 
\| \{\sv \cdot \nn\}   \|^2_{0,\Sigma}
+  \alpha\| [\sv_{\tau} ] \|^2_{0,\Sigma}
\\
\leqslant
  \|\sv^+\|_{0,\Omega_{+}}   \|\sg^+\|_{0,\Omega_{+}}
+ \|\sv^-\|_{0,\Omega_{-}}   \|\sg^-\|_{0,\Omega_{-}}  +\| [\sv  ] \|_{0,\Sigma} \|\sh \|_{0,\Sigma}
  +\| \{\sv\} \|_{0,\Sigma} \|\sl \|_{0,\Sigma}
  \ .
\end{multline*}
Let $\eps>0$ such that $\alpha\leqslant \frac{1}{2\eps}$. Then, for all $\eta_{1},\eta_{2}>0$, there holds
\begin{multline}
\label{E2}
2 \eps \|\bd(\sv^-) \|_{0,\Omega_{-}}^2
  + \kappa\| \sv^- \|^2_{0,\Omega_{-}}  
  + 2 \mu C^2 \|\sv^+ \|_{1,\Omega_{+}}^2
+  \frac{1}{2\eps} \| [\sv  \cdot \nn ] \|^2_{0,\Sigma}
+\beta 
\| \{\sv \cdot \nn\}   \|^2_{0,\Sigma}
+  \alpha\| [\sv ] \|^2_{0,\Sigma}
\\
\leqslant
  \frac{\eta_{1}}{2}  \|\sv\|^2_{0,\Omega} + \frac1{2\eta_{1}}  \|\sg\|^2_{0,\Omega}
+\frac{\eta_{2}}{2} \| [\sv  ] \|^2_{0,\Sigma} +\frac1{2\eta_{2}}   \|\sh \|^2_{0,\Sigma}
  +\| \{\sv \cdot\nn\}\|_{0,\Sigma} \|\sl\|_{0,\Sigma}+\| \{\sv_{\tau}\} \|_{0,\Sigma} \|\sl \|_{0,\Sigma}
  \ .
\end{multline} 
We treat hereafter the last two terms : for all $\eta_{3}>0$, there holds
$$
\| \{\sv \cdot\nn\}\|_{0,\Sigma} \|\sl\|_{0,\Sigma}
\leqslant  \frac{\eta_{3}}{2}  \| \{\sv \cdot\nn\}\|^2_{0,\Sigma}
+ \frac1{2\eta_{3}} \|\sl\|^2_{0,\Sigma} \ .
$$
Since $\{\sv_{\tau}\}=\sv^+_{\tau}-\frac12 [\sv_{\tau} ]$, there holds
$$
\| \{\sv_{\tau}\} \|_{0,\Sigma} \|\sl \|_{0,\Sigma}
\leqslant 
(\|\sv^+_{\tau}\|_{0,\Sigma} + \frac12\| [\sv_{\tau} ] \|_{0,\Sigma} )  \|\sl \|_{0,\Sigma}\ ,
$$ 
hence, using a trace inequality, there exists $c>0$ such that 
$$
\| \{\sv_{\tau}\} \|_{0,\Sigma} \|\sl \|_{0,\Sigma}
\leqslant 
(c \|\sv^+\|_{1,\Omega_{+}} + \frac12\| [\sv_{\tau} ] \|_{0,\Sigma} )  \|\sl \|_{0,\Sigma}\ .
$$ 
We infer : for all $\eta_{4}, \eta_{5}>0$, 
$$
  \| \{\sv_{\tau}\} \|_{0,\Sigma} \|\sl \|_{0,\Sigma}
\leqslant 
 \frac{\eta_{4} c}{2}  \|\sv^+\|^2_{1,\Omega_{+}} 
 +  \frac{c}{2\eta_{4}}  \| \sl  \|^2_{0,\Sigma} 
 +\frac{\eta_{5}}{4}  \| [\sv_{\tau} ] \|^2_{0,\Sigma}
+ \frac1{4\eta_{5}} \|\sl\|^2_{0,\Sigma} 
 \ .
$$
According to \eqref{E2}, we obtain
\begin{multline}
\label{E3}
 2 \eps \|\bd(\sv^-) \|_{0,\Omega_{-}}^2
  + \kappa\| \sv^- \|^2_{0,\Omega_{-}}  
  + 2 \mu C^2 \|\sv^+ \|_{1,\Omega_{+}}^2
+  \frac{1}{2\eps} \| [\sv  \cdot \nn ] \|^2_{0,\Sigma}
+  \alpha\| [\sv ] \|^2_{0,\Sigma}
+\beta 
\| \{\sv \cdot \nn\}   \|^2_{0,\Sigma}
\\
\leqslant
  \frac{\eta_{1}}{2}  \|\sv\|^2_{0,\Omega} 
  +\frac{\eta_{2}}{2} \| [\sv  ] \|^2_{0,\Sigma}
  +\frac{\eta_{3}}{2}  \| \{\sv \cdot\nn\}\|^2_{0,\Sigma}
 + \frac{\eta_{4} c }{2}  \|\sv^+\|^2_{1,\Omega_{+}} 
 +\frac{\eta_{5}}{4}  \| [\sv_{\tau} ] \|^2_{0,\Sigma}
 \\
  + \frac1{2\eta_{1}}  \|\sg\|^2_{0,\Omega}
+\frac1{2\eta_{2}}   \|\sh \|^2_{0,\Sigma}
+ (\frac1{2\eta_{3}} +  \frac{c}{2\eta_{4}} + \frac1{4\eta_{5}}) \|\sl\|^2_{0,\Sigma} 
  \ .
\end{multline} 
We fix constants $\eta_{i}$ such that 
$$ \frac{\eta_{4} c }{2}\leqslant \frac{ \mu C^2}{3}\ ,
\quad \frac{\eta_{1}}{2} \leqslant \min(\frac {\kappa}{2},\frac{ \mu C^2}{3}) \ ,
\quad  \frac{\eta_{2}}{2} + \frac{\eta_{5}}{4} \leqslant \frac{\alpha}{2} \ ,
\quad  \frac{\eta_{3}}{2} \leqslant \frac{\beta}{2} \ .
$$ 
According to \eqref{E3}, we infer
\begin{multline}
\label{E4}
 \eps \|\bd(\sv^-) \|_{0,\Omega_{-}}^2
  +  \frac {\kappa}{4} \| \sv^- \|^2_{0,\Omega_{-}}  
  + \frac{ \mu C^2}{3} \|\sv^+ \|_{1,\Omega_{+}}^2
+  \frac{1}{4\eps} \| [\sv  \cdot \nn ] \|^2_{0,\Sigma}
+  \frac{\alpha}{4} \| [\sv ] \|^2_{0,\Sigma}
+\frac{\beta}{4}
\| \{\sv \cdot \nn\}   \|^2_{0,\Sigma}
\\
\leqslant
C(\mu, \kappa, \alpha, \beta)
\left(  \|\sg\|^2_{0,\Omega}
+  \|\sh \|^2_{0,\Sigma} 
+ \|\sl\|^2_{0,\Sigma} 
\right)
  \ .
\end{multline} 

\qed

\section{Formal Asymptotic Expansion}
\label{AE}

\begin{hyp}
From now on, we assume that the surfaces $\Sigma$ (interface) and $\Gamma$ (external boundary) are smooth. 
\end{hyp}

\begin{notn}
If $X$ is a vector field defined on $\Sigma$, we denote by $X_\tau$ the tangent components of $X$: $\dsp X_\tau(x)=X(x)-(X(x)\cdot \nn(x))\nn(x)$, so that $X=(X\cdot \nn)\nn + X_\tau$. For example, we  denote by $(\bd(\sv_{\eps}).\nn)_\tau$ the tangent components of the normal constraint $\bd(\sv_{\eps}).\nn$ defined on the interface  $\Sigma$.
\end{notn}

Rewriting the
 transmission conditions set on $\Sigma$ in \eqref{0SBBJ}, we use Formulation \eqref{putain} to obtain the following equivalent problem,

\begin{align}
 \label{Das1}  & - \eps \Delta \sv_{\eps}^{-} + \nabla \sp_{\eps}^-+\kappa\sv_{\eps}^{-}=\sg^-
& \mbox{in}& \quad\Omega_{-}
\\[0.5ex]
 \label{Das2} & -\nabla \cdot  \bsigma^+(\sv_{\eps}^+,\sp_{\eps}^+ )=\sg^+ 
 & \mbox{in}&\quad\Omega_{+}
\\[0.5ex]
 \label{Das3} & \nabla\cdot \sv_{\eps}^-=0 &\mbox{in}&\quad\Omega_{-}
\\[0.5ex]
 \label{Das4}  &\nabla\cdot \sv_{\eps}^+=0 &\mbox{in}&\quad\Omega_{+}
\\[0.5ex]  
 \label{Das5} &2 \mu \bd (\sv_{\eps}^+)\cdot \nn -\sp_{\eps}^+ \nn =2\eps \bd (\sv_{\eps}^-)\cdot \nn - \sp_{\eps}^-\nn +\frac{\beta}{2} \left( ( \sv_{\eps}^++\vepsm)\cdot \nn\right) \nn + \sl & \mbox{on}& \quad\Sigma
 \\[0.5ex] 
 \label{Das7}  & \alpha(\sv_{\eps}^+ -\vepsm)_\tau = 2\mu \left(\bd(\sv_\eps^+)\cdot \nn \right)_{\tau}  -(\sh+\frac 1 2 \sl)_\tau 
& \mbox{on}&\quad \Sigma   
    \\[0.5ex] 
     \label{Das8}  &  \eps^{-1} (\sv_{\eps}^+ -\vepsm)\cdot \nn =  2 \mu(\bd(\sv_\eps^+)\cdot \nn)\cdot\nn -\sp_{\eps}^+ -\frac{\beta}{4}( \sv_{\eps}^++\vepsm)\cdot \nn -(\sh+\frac 1 2 \sl) \cdot \nn
&  \mbox{on}&\quad \Sigma  \ 
    \\[0.5ex] 
    \label{Das9}  &  \sv^+_{\eps} =  0
  \quad &\mbox{on}&\quad \Gamma    
\end{align}
(we recall that $\dsp \nabla \cdot  \bsigma^+(\sv_{\eps}^+,\sp_{\eps}^+ )=\mu\Delta \sv_{\eps}^+- \nabla \sp_{\eps}^+ $)

\vspace{2mm}

As already said, when $\eps$ tends to zero, we formally converge to Problem \eqref{SD-BJC}. The well posedness of this limit system is establish in Section \ref{secregprof}. This problem is incompatible with the limit of the jump condition \eqref{Das7} on the tangential velocity:
$$
 \alpha(\sv_{0}^+ -\sv_{0}^-)_\tau = 2 \mu \left(\bd(\sv_0^+)\cdot \nn \right)_{\tau}  -(\sh+\frac 1 2 \sl)_\tau \quad 
 \mbox{on}\quad \Sigma  \ . $$
Hence, it appears a boundary layer that we describe with a multiscale method, namely a WKB expansion: 
we exhibit series expansions in powers of $\sqrt{\eps}$ for the flow $\sv_{\eps}$, and the pressure $\sp_{\eps}$. In the fluid part, we do not expect the formation of a boundary layer, therefore we look for an {\it Ansatz} on the form:\begin{gather}
\label{Esv+}
   \vepsp(x)  \approx \sum_{j\geqslant0} {\eps}^{\frac j 2} \Vplus_j(x) \, ,
   \\
    \label{Esp+}
   \sp^+_{\eps}(x)  \approx \sum_{j\geqslant0} {\eps}^{\frac j 2}\Pplus_j(x)\,.
   \end{gather}
   
   In the porous part, we  denote by $d: \Omega_{-}\to \R^+$ the euclidean distance to $\Sigma$, $d(x) = \textrm{dist} (x, \Sigma)$.     We  describe the velocity and the pressure on the following way:  
    \begin{gather}
   \label{Esv-}
   \vepsm(x)  \approx   \sum_{j\geqslant0} {\eps}^{\frac j 2} \sv^-_j(x,\frac{d(x)}{\sqrt \eps}) \, ,
 \quad\mbox{with}\quad   
   \sv^-_j(x,\sz) =
\Vmb_{j}(x)+ \Vmt_j(x,\sz)\, , \\
  \label{Esp-}
   \sp^-_{\eps}(x)  \approx \sum_{j\geqslant0} {\eps}^{\frac j 2} \sp^-_j(x,\frac{d(x)}{\sqrt \eps}) \ ,
\quad\mbox{with}\quad   
    \sp^-_j(x,\sz) =
   \pmob_{j}(x)     + \fp_j(x,\sz)\, .
\end{gather}
The terms $ \Vmt_j$ and $ \fp_j$ are boundary layer terms and defined on $\Sigma\times\R^+$. They are required to  tend to $0$ (such as their derivatives) when $\sz\to \infty$.

 We use the following notations (see also \cite{CF03}) : $x=(x_{1},x_{2}, x_{3})$ denotes the cartesian coordinates in $\R^3$, 
 $$ 
 \partial_{i}=\frac{\partial }{\partial x_{i}}\ ,\quad \partial_\sz=\frac{\partial }{\partial \sz}\ ,\quad \partial_{\sz \sz}=\frac{\partial^2}{\partial \sz ^2}
 ,\quad 
\nabla \fp =(\partial_1 \fp , \partial_2 \fp, \partial_3 \fp)^T. $$
 For  $\dsp  \fv=(\widetilde{\mathrm{v}}^{-,1},\widetilde{\mathrm{v}}^{-,2},\widetilde{\mathrm{v}}^{-,3})^T,$
$$
 \nabla \cdot \fv =\partial_1 \widetilde{\mathrm{v}}^{-,1}+ \partial_2 \widetilde{\mathrm{v}}^{-,2} +\partial_3 \widetilde{\mathrm{v}}^{-,3} 
, \quad ( \bd (\fv))_{ij}= \frac{1}{2} \left( \dd{\widetilde{\mathrm{v}}^{-,j}}{x_i} + \dd{\widetilde{\mathrm{v}}^{-,i}}{x_j}\right)
 , $$
 $$ (\nabla d\cdot \nabla)\dz \fv=\sum_{j=1}^3 \dd{d}{x_j}\frac{\partial^2  \widetilde{\mathrm{v}}^{-}}{\partial x_j\partial\sz }
 , \quad \Delta \fv = \partial_{1}^2 \fv + \partial_{2}^2 \fv+\partial_{3}^2 \fv.$$

We recall that $\nn$ is the outward unit normal on $\partial \Omega_{-}=\Sigma$.  Since we assume that $\Sigma$ is a regular manifold, then $d$ is smooth in a neighborhood of $\Sigma$. In this neighborhood, $\vert \nabla d\vert=1$. In addition, for $x\in \Sigma$, $\nabla d (x)= -\nn(x)$ so that $\partial_\nn d = -1$ on $\Sigma$. We extend $\nn$ by setting $\nn(x)=-\nabla d(x)$.  If $X$ is a vector field defined in $\Omega_{-}$, in a neighborhood of $\Sigma$, we define the tangential part of $X$ by $\dsp X_\tau (x)=X(x)-(X(x)\cdot \nn(x))\nn(x)$.

\vspace{2mm}

By simple calculations there holds
$$\begin{array}{rl}\dsp 
\Delta \left(x\mapsto \Vmt(x,\frac{d(x)}{\sqrt \eps})\right) = & \dsp \eps^{-1} \partial_{\sz \sz}\fv (x,\frac{d(x)}{\sqrt \eps})
+\eps^{-\frac 1 2}\left(2(\nabla d \cdot  \nabla)  \partial_\sz \fv + \Delta d \ \dz\fv\right)(x,\frac{d(x)}{\sqrt \eps})\\
&\dsp 
+\Delta\Vmt (x,\frac{d(x)}{\sqrt \eps})
\end{array}$$
$$
\nabla \left(x\mapsto\fp(x;\frac{d(x)}{\sqrt \eps})\right) =\eps^{-\frac 1 2}\nabla d \ \dz\fp(x,\frac{d(x)}{\sqrt \eps})+ \nabla\fp (x,\frac{d(x)}{\sqrt \eps})
$$
$$
\nabla \cdot \left(x\mapsto\Vmt(x;\frac{d(x)}{\sqrt \eps})\right) = \eps^{-\frac 1 2}\nabla d \cdot \dz \Vmt(x,\frac{d(x)}{\sqrt \eps})+ \nabla\cdot \Vmt(x,\frac{d(x)}{\sqrt \eps}) 
$$
$$\bd\left(x\mapsto\Vmt(x,\frac{d(x)}{\sqrt \eps})\right)  \cdot \nn =-\eps^{-\frac 1 2} \frac12\left( \dz \Vmt + (\dz \Vmt\cdot \nn)\nn\right) (x;\frac{d(x)}{\sqrt \eps}) +\bd( \Vmt) \cdot \nn (x;\frac{d(x)}{\sqrt \eps})  \quad \mbox{on} \quad \Sigma\ .$$
We insert the {\it Ansatz} \eqref{Esv+}-\eqref{Esp+}-\eqref{Esv-}-\eqref{Esp-} in equations \eqref{Das1}-\eqref{Das9} and we perform the identification of terms with the same power in $\sqrt\eps$.   

\subsection{Asymptotic expansion of  \eqref{Das1}}

\paragraph*{Order $\eps^{-\frac{1}{2}}$} There holds $\nabla d \ \dz \fp_{0}=0$, and since  $\fp_{0}\to 0$ when $\sz \to \infty$, we infer 
\begin{equation}
\label{fp0}
\fp_{0}=0 \quad \mbox{in} \quad \Omega_{-}\times \R^+\ .
\end{equation}

\paragraph*{Order $\eps^{0}$} According to \eqref{fp0}, $(\sv^-_{0},\sp_{0}^-)$ solves  
\begin{equation}
\label{sv0sp0}
-\dzz\fv_{0} + \kappa \Vmt_{0} +\kappa \Vmb_{0}+\nabla\pmob_{0} +\nabla d \ \dz\fp_{1}=\sg^-
 \quad \mbox{in} \quad \Omega_{-}\times \R^+\ .
\end{equation}
 Hence, taking the limit of this equation when $\sz\to \infty$, since $\Vmt_0$ and all its derivatives tends to zero when $\sz$ tends to $+\infty$, we infer 
\begin{equation}
\label{sv0b}
 \kappa \Vmb_{0}  +\nabla\pmob_{0} =\sg^- \quad \mbox{in} \quad \Omega_{-}\ ,
\end{equation}
and by difference with the previous equation, we obtain
\begin{equation}
\label{fv0}
-\dzz \fv_{0} + \kappa \fv_{0}  +\nabla d \ \dz\fp_{1}=0 \quad \mbox{in} \quad \Omega_{-}\times \R^+\ .
\end{equation}

\paragraph*{Order $\eps^{\frac{1}{2}}$}  $(\sv^-_{1},\sp_{1}^-)$ solves 
\begin{equation}
\label{v1p1-}
-\dzz \fv_{1} + \kappa \Vmt_{1} +\kappa \Vmb_{1} +\nabla\pmob_{1}(x)   +\nabla\fp_{1} +\nabla d \ \dz\fp_{2}=( 2 (\nabla d \cdot  \nabla)  \dz \fv    _{0} + \Delta d \ \dz\fv_{0}) \ ,
\end{equation}
and taking the limit of this equation when $\sz\to \infty$, we infer 
\begin{equation}
\label{sv1b}
 \kappa \Vmb_{1}(x)  +\nabla\pmob_{1}(x)  =0 \quad \mbox{in} \quad \Omega_{-}\ .
\end{equation}
By difference with the previous equation, we obtain
\begin{equation}
\label{fv1}
-\dzz \fv_{1} + \kappa \fv_{1}    +\nabla\fp_{1} +\nabla d \ \dz\fp_{2}= 2 (\nabla d \cdot  \nabla)  \fv    _{0,Z} + \Delta d \  \dz \fv    _{0}  \ , \quad \mbox{in} \quad \Omega_{-}\times \R^+\ .
\end{equation}

\paragraph*{Order $\eps^{\frac{j}{2}}$, $j\geqslant 2$ }  $(\sv^-_{j},\sp_{j}^-)$ solves 
\begin{multline}
\label{vj-}
-\dzz \fv_{j} + \kappa \Vmt_{j} +\kappa \Vmb_{j}  +\nabla\pmob_{j} +\nabla\fp_{j} +\nabla d \ \dz\fp_{j+1}
\\
=( 2 (\nabla d \cdot  \nabla)  \dz \fv    _{j-1} + \Delta d \ \dz\fv_{j-1}) +   \Delta (\Vmb_{{j-2}}+\fv_{j-2}) \ ,
\end{multline}
hence, we infer
\begin{equation}
\label{svjb}
 \kappa \Vmb_{j}  +\nabla\pmob_{j}  = \Delta\Vmb_{{j-2}} \quad \mbox{in} \quad \Omega_{-}\ ,
\end{equation}
\begin{multline}
\label{svjt}
-\dzz \fv_{j} + \kappa \fv_{j}  +\nabla\fp_{j} +\nabla d \ \dz \fp_{j+1}
\\
=( 2 (\nabla d \cdot  \nabla)  \dz \fv    _{j-1} 
+ \Delta d \ \dz \fv_{j-1}) +   \Delta \fv_{j-2}  \quad \mbox{in} \quad \Omega_{-}\times \R^+\ .
\end{multline}

\subsection{Asymptotic expansion of \eqref{Das2}}

\paragraph*{Order $\eps^{0}$}  $ \Vplus_{0} $ satisfies 
\begin{equation}
\label{sv+0}
- \nabla \cdot  \bsigma^+( \Vplus_{0} , \Pplus_{0})=\sg^+ \quad \mbox{in} \quad \Omega_{+}\ .
\end{equation}

\paragraph*{Order $\eps^{\frac j 2}$, $j\geqslant 1$} 
For all $j\in\N$, $ \Vplus_{j} $ satisfies 
\begin{equation}
\label{sv+j}
-  \nabla \cdot \bsigma^+(  \Vplus_{j} , \Pplus_{j})=0 \quad \mbox{in} \quad \Omega_{+}\ .
\end{equation}

\subsection{Asymptotic expansion of \eqref{Das3}}

\paragraph*{Order $\eps^{-\frac{1}{2}}$} There holds 
\begin{equation}
\label{fv0nbis}
\nabla d \cdot \dz \fv_{0}=0 \quad \mbox{in} \quad \Omega_{-}\times \R^+\ .
\end{equation}
Since $\nabla d = -\nn $, since $\Vmt_0$ tends to zero when $\sz$ tends to $+\infty$, we obtain 
\begin{equation}
\label{fv0n}
\fv_{0}\cdot \nn=0 \quad \mbox{in} \quad \Omega_{-}\times \R^+\ .
\end{equation}

\paragraph*{Order $\eps^{\frac{j}{2}}$, $j \geqslant 0$} 
There holds 
\begin{equation}
\label{fvjnbis}
\nabla d \cdot  \dz \fv_{j+1}+ \nabla\cdot\Vmb_{j}  +\nabla\cdot \fv_j=0 \quad \mbox{in} \quad \Omega_{-}\times \R^+\ .
\end{equation}
We infer  
\begin{gather}
\label{divvbj}
 \nabla \cdot\Vmb_{j} =0 \quad \mbox{in} \quad \Omega_{-} \, ,
\\  
\label{divfvj}
  \nabla d \cdot  \dz \fv_    {j+1} +\nabla\cdot \fv_j=0 \quad \mbox{in} \quad \Omega_{-}\times \R^+ 
  \, .
\end{gather}

\subsection{Asymptotic expansion of  \eqref{Das4}}

\paragraph*{Order $\eps^{\frac{j}{2}}$} For all $j\in\N$, $ \Vplus_{j} $ satisfies 
\begin{equation}
\label{divsv+j}
\nabla \cdot \Vplus_{j} =0 \quad \mbox{in} \quad \Omega_{+}\ .
\end{equation}

\subsection{Asymptotic expansion of  \eqref{Das5}}

\paragraph*{Order $\eps^{0}$.}

\begin{equation}
\label{dnsv0+}
2\mu \bd( \Vplus_0) \cdot \nn - \Pplus_0 \nn = -\pmob_{0}\nn - \fp_{0}(\cdot , 0)\nn  +\frac \beta 2\left((\Vplus_{0}+\Vmb_{0}+\Vmt_0(\cdot , 0))\cdot \nn\right) \nn +\sl  \quad \mbox{on} \quad \Sigma 
\end{equation}

\paragraph*{Order  $\eps^{\frac{1}{2}}$.}

\begin{equation}
\label{dnsv1t}
\begin{array}{rl}
\dsp 
2\mu \bd (\Vplus_1)\cdot \nn - \Pplus_1\nn = &  \dsp - \frac12 \left(\dz\fv_{0}(\cdot , 0) + (\dz\fv_{0}(\cdot , 0) \cdot \nn) \nn \right)  -\pmob_{1}\nn - \fp_{1}(\cdot , 0) \nn 
\\
&\dsp +\frac \beta 2\left( (\Vplus_{1}+\Vmb_{1}+\Vmt_1(\cdot , 0))\cdot \nn \right) \nn \quad \mbox{on} \quad \Sigma\end{array}
\end{equation}

\paragraph*{Order $\eps^{\frac{j}{2}}$, $j\geqslant 2$.}
\begin{equation}
\label{dnsvjt}
\begin{array}{rl}
\dsp 2\mu \bd (\Vplus_j) \cdot \nn  -\Pplus_{j} \nn          =   & \dsp -\frac12 \left( \dz\fv_{j-1}(\cdot,0) + (\dz\fv_{j-1}(\cdot,0)\cdot \nn) \nn \right)  +2 \bd( \Vmb_{j-2})\cdot \nn\\
\\
&\dsp +2\bd(\Vmt_{j-2})(\cdot ,0)\cdot \nn 
\dsp -\pmob_{j}\nn-\fp_{j}(\cdot , 0)\nn   \\ \\
\dsp &\dsp + \frac \beta 2\left((\Vplus_{j}+\Vmb_{j}+\Vmt_j(\cdot , 0))\cdot \nn\right)\nn \
 \quad \mbox{on}\quad \Sigma  \, .
 \end{array}
\end{equation}

\subsection{Asymptotic expansion of  \eqref{Das7}}

\paragraph*{Order $\eps^{0}$} 
\begin{equation}
\label{sv0tz=0}
 \alpha(\Vplus_{0} -\Vmb_0-\Vmt_0(\cdot , 0))_\tau =2 \mu\left(\bd(\Vplus_{0})\cdot \nn \right)_\tau -\left(\sh +\frac 1 2 \sl\right)_\tau \quad \mbox{on}\quad \Sigma   
\end{equation}

\paragraph*{Order $\eps^{\frac{j}{2}}$, $j\geqslant 1$} 
\begin{equation}
\label{svjtz=0}
 \alpha(\Vplus_{j} -\Vmb_{j}-\Vmt_{j}(\cdot , 0))_\tau =2 \mu\left(\bd(\Vplus_{j})\cdot \nn\right)_\tau \quad \mbox{on}\quad \Sigma   
\end{equation}

\subsection{Asymptotic expansion of \eqref{Das8}}

\paragraph*{Order $\eps^{0}$} 

\begin{equation}
\label{sv0n}
(\Vplus_{0} -\Vmb_0-\Vmt_0(\cdot , 0))\cdot \nn= 0
 \quad \mbox{on}\quad \Sigma
\end{equation}

\paragraph*{Order $\eps^{\frac{1}{2}}$} 

\begin{equation}
\label{sv1n}
(\Vplus_{1} -\Vmb_1-\Vmt_1(\cdot , 0))\cdot \nn= 0
 \quad \mbox{on}\quad \Sigma
\end{equation}

\paragraph*{Order $\eps$}

\begin{equation}
\label{sv2n}
 (\Vplus_{2} -\Vmb_2-\Vmt_2(\cdot , 0))\cdot \nn =   2\mu\left(\bd(\Vplus_{0})\cdot \nn \right)\cdot\nn -\Pplus_{0} -\frac \beta 4 \left (\Vplus_{0} +\Vmb_{0}+\Vmt_{0}(\cdot , 0)\right)\cdot \nn -(\sh+ \frac 1 2 \sl)\cdot \nn 
  \quad \mbox{on}\quad \Sigma    \ 
\end{equation}

\paragraph*{Order $\eps^{\frac{j}{2}}$, $j > 2$}

\begin{equation}
\label{svjn}
 (\Vplus_{j} -\Vmb_j-\Vmt_j(\cdot , 0))\cdot \nn =  2 \mu\left(\bd(\Vplus_{j-2})\cdot \nn \right)\cdot\nn -\Pplus_j -\frac{ \beta}{4}\left (\Vplus_{j-2} +\Vmb_{j-2}+\Vmt_{j-2}(\cdot , 0)\right)\cdot \nn 
  \quad \mbox{on}\quad \Sigma    \ .
\end{equation}

\subsection{Asymptotic expansion of  \eqref{Das9}}

\paragraph*{Order $\delta^{j}$}  For all $j\in\N$, $ \Vplus_{j} $ satisfies the external boundary condition
\begin{equation}
\label{v+j}
\Vplus_{j} =0 \quad \mbox{in} \quad \Gamma\ .
\end{equation}

\section{Determination, existence and regularity of the asymptotics}
\label{SDERA}

\vspace{2mm}
Let $\bk\geqslant 1$. We assume that the data in \eqref{Das1}-\eqref{Das9} satisfy:
\begin{equation}
\label{regdata}
\sg^-\in \bH^\bk(\Omega_{-}),\quad   \sg^+\in \bH^{\bk-1}(\Omega_{+}),\quad  \sl\in \bH^{\bk-\frac{1}{2}}(\Sigma)\quad\mbox{and}\quad\sh\in \bH^{\bk-\frac{1}{2}}(\Sigma).
\end{equation}

\subsection{Equations satisfied by asymptotics of order 0}
\label{sec-as0}
\subsubsection{Determination of $ \fp_{0}$ and $ \fv_{0}\cdot\nn$}

We have already obtained in  \eqref{fp0} and $\eqref{fv0n}$ that 
\begin{equation}
\label{Efp0}
\fp_{0}=0 \mbox{ and }  \fv_{0}\cdot \nn =0 \quad \mbox{in} \quad \Omega_{-}\times \R^+ \ .
\end{equation}

\subsubsection{Determination of $\mathrm{v}^-_{0}$, $\mathrm{p}^-_{0}$  $\Vplus_{0}$ and $\Pplus_{0}$}

According to \eqref{sv0n} and \eqref{fv0n}, we infer 
$$
 (\Vplus_{0} -\Vmb_{0} )\cdot \nn=0 \quad \mbox{on} \quad \Sigma \ .
 $$ 

According to  and \eqref{fp0} and \eqref{dnsv0+}
$$
2 \mu \bd(\Vplus_{0})\cdot\nn -\Pplus_{0}\nn = -\pmob_{0}\nn + \frac{\beta}{2}\left((\Vplus_{0}+\Vmb_{0})\cdot \nn\right)\nn + \sl \quad \mbox{on} \quad \Sigma\ .
$$

Finally, according to \eqref{sv0b}, \eqref{divvbj}, \eqref{dnsv0+}, and according to \eqref{sv+0} and \eqref{divsv+j} when $j=0$, $(\mathrm{v}^-_{0},\Vplus_{0})$, and $(\mathrm{p}^-_{0},\Pplus_{0})$ satisfy the transmission problem :
\begin{equation}
\label{SBBJ0}
 \left\{
   \begin{array}{lll}
    \kappa\Vmb_{0} + \nabla\pmob_{0}=\sg^-
 \quad&\mbox{in}\quad \Omega_{-}
\\[0.5ex]
\nabla\cdot\Vmb_{0}=0 \quad&\mbox{in}\quad \Omega_{-}\ .
\\[0.5ex]
 - \nabla \cdot \bsigma^+( \Vplus_{0} , \Pplus_{0})=\sg^+
 \quad&\mbox{in}\quad \Omega_{+}
\\[0.5ex]
\nabla\cdot \Vplus_{0}=0 \quad&\mbox{in}\quad \Omega_{+}
\\[0.5ex]
(\Vplus_{0} -\Vmb_{0} )\cdot \nn=0 \quad &\mbox{on}\quad \Sigma
 \\[0.5ex]
 2\mu\bd(\Vplus_{0})\cdot\nn -\Pplus_{0}\nn = -\pmob_{0}\nn + \frac{\beta}{2}\left((\Vplus_{0}+\Vmb_{0})\cdot \nn\right)\nn + \sl
 \quad &\mbox{on} \quad \Sigma
\\[0.2ex]
\Vplus_{0}=0
  \quad &\mbox{on}\quad \Gamma \ .
   \end{array}
    \right.
\end{equation}

Let us claim the following existence and regularity result concerning such a problem

\begin{prop}
\label{profilbarre}
Let $k\geq 1$.  Let $g^-\in \bH^k(\Omega_{-})$, $g^+\in \bH^{k-1}(\Omega_{+})$, $l\in \bH^{k-\frac{1}{2}}(\Sigma)$ and $h\in \bH^{k-\frac{1}{2}}(\Sigma)$. We assume that $h$ satisfies the compatibility condition 
$$\int_\Sigma h \,\dr \sigma=0.$$
We consider the following problem:
\begin{equation}
\label{SBBJ}
 \left\{
   \begin{array}{lll}
    \kappa\Vmb + \nabla\pmob=g^-
 \quad&\mbox{in}\quad \Omega_{-}
\\[0.5ex]
\nabla\cdot\Vmb=0 \quad&\mbox{in}\quad \Omega_{-}\ 
\\[0.5ex]
 -\nabla \cdot \bsigma^+( \Vplus , \Pplus)= g^+
 \quad&\mbox{in}\quad \Omega_{+}
\\[0.5ex]
\nabla\cdot \Vplus= 0\quad&\mbox{in}\quad \Omega_{+}
    \\[0.5ex]
(\Vplus -\Vmb)\cdot \nn= h
\quad &\mbox{on}\quad \Sigma
 \\[0.5ex]
  2\mu\bd(\Vplus)\cdot\nn -\Pplus\nn = -\pmob \nn + \frac{\beta}{2}\left((\Vplus+\Vmb)\cdot \nn\right)\nn + l
 \quad &\mbox{on} \quad \Sigma
\\[0.2ex]
\Vplus=0
  \quad &\mbox{on}\quad \Gamma \ .
   \end{array}
    \right.
\end{equation}
Then \eqref{SBBJ} admits a weak solution unique up to additive constants for $\Pplus$ and $\pmob$.  It satisfies the regularity properties:
$$\Vplus\in \bH^{k+1}(\Omega_{+}), \quad \Vmb\in \bH^{k}(\Omega_{-}), \quad \Pplus\in \H^{k}(\Omega_{+})\quad\mbox{and} \quad \pmob\in \H^{k+1}(\Omega_{-}).$$
\end{prop}

This proposition will be proved in Section \ref{secregprof}.

\vspace{2mm}
Using Proposition \ref{profilbarre} with $g^\pm=\sg^\pm$, $l=\sl$ and $h=0$, using the regularity of the data \eqref{regdata}, we obtain the existence and uniqueness (up to an additive constant for the pressures) of the asymptotics $\Vplus_0\in \bH^{\bk+1}(\Omega_{+}), \quad \Vmb_0\in \bH^{\bk}(\Omega_{-}), \quad \Pplus_0\in \H^{\bk}(\Omega_{+})\quad\mbox{and} \quad \pmob_0\in \H^{\bk+1}(\Omega_{-}).$

\subsubsection{Determination of $(\Vmt_{0})_\tau$}
The tangential component of  \eqref{fv0} reduce to 
\begin{equation*}
-\dzz(\Vmt_{0})_\tau  + \kappa (\Vmt_{0})_\tau =0 \quad \mbox{in} \quad \Omega_{-}\times \R^+  \ .
\end{equation*}
 Hence, since $\fv_{0}\to 0$ when $\sz\to \infty$ , we obtain
$$
\fv_{0}( x, \sz)_\tau=\fv_{0}(x,0)_\tau \exp( -\sqrt {\kappa}\,\sz )  \quad \mbox{in} \quad \Omega_{-}\times \R^+\ ,
$$
According to \eqref{sv0tz=0}, we infer
$$
\fv_{0}( x, 0)_\tau= (\Vplus_{0} -\Vmb_{0})_\tau -\frac{2\mu}{\alpha} (\bd( \Vplus_0)\cdot \nn)_\tau +\frac{1}{\alpha} \left(\sh+\frac 1 2 \sl\right)_\tau \quad \mbox{on}\quad \Sigma\ .
$$ 
Hereafter, we introduce $\sw_{0}^0$ which is a tangential extension of $(\Vplus_{0} -\Vmb_{0})_\tau -\frac{2\mu}{\alpha} (\bd( \Vplus_0)\cdot \nn)_\tau +\frac{1}{\alpha} \left(\sh+\frac 1 2 \sl\right)_\tau$ in the domain $\Omega_{-}$. We choose this extension $\sw_{0}^0$ such that it has a support in a tubular neighborhood of the interface $\Sigma$.  Since $ \Vplus_{0} \in \bH^{\sk+1}(\Omega_{+})$, since $\Vmb_{0} \in \bH^{\sk }(\Omega_{-})$, since $\sl$ and $\sh$ belong to the space $\bH^{\sk-\frac 1 2 }(\Sigma)$, we can take this extension satisfying    $\sw_{0}^0\in \bH^{\sk }(\Omega_{-})$, and using \eqref{Efp0} we obtain
\begin{equation}
\label{fv0-}
\fv_{0}( x, \sz)=\sw_{0}^0(x) \exp( -\sqrt {\kappa}\,\sz )  \quad \mbox{in} \quad \Omega_{-}\times \R^+\ ,
\end{equation}
and then $\fv_{0}$ is completely defined by \eqref{fv0-}.


\subsection{Equations satisfied by asymptotics of order 1}
\label{sec-as1}
\subsubsection{Determination of $ \fp_{1}$}

According to \eqref{fv0n}, by taking the normal components in \eqref{fv0}, we obtain $\fp_{1,Z}=0$. Hence, since $\fp_{1}\to 0$ when $\sz\to \infty$, we obtain
\begin{equation}
\label{fp1}
\fp_{1}=0  \quad \mbox{in} \quad \Omega_{-}\times \R^+ 
\ .
\end{equation}

\subsubsection{Determination of $ \Vmt_{1}\cdot \nn$}

 According to \eqref{divfvj}
$$
\dz \fv_{1}\cdot \nabla d=- \nabla\cdot \fv_{0} =- (\nabla \cdot \sw_0^0) \exp( -\sqrt {\kappa}\,\sz )
\quad \mbox{in} \quad \Omega_{-}\times \R^+\ .
$$
Hence, since $\Vmt_1$ tends to zero when $\sz$ tends to $+\infty$,
\begin{equation}
\label{fv1n}\dsp \Vmt_{1}\cdot\nn= \nu_1^0 \exp( -\sqrt {\kappa}\,\sz)\quad \mbox{with}\quad
\nu_1^0=\frac{-1}{\sqrt\kappa }(\nabla\cdot \sw_0^0)\in \bH^{\bk- 1  }(\Omega_{-}).\end{equation}



\subsubsection{Determination of $(\Vmb_{1},\Vplus_{1})$, and $(\pmob_{1},\Pplus_{1})$}

According to \eqref{fv0n} and \eqref{fp1}, the transmission condition \eqref{dnsv1t} writes :
\begin{equation*}
2 \mu\bd(\Vplus_{1})\cdot\nn -\Pplus_{1}\nn =  -\frac12 \dz \Vmt_0 -\pmob_{1}\nn +\frac{\beta}{2} \left((\Vplus_{1} +\Vmb_{1} +\Vmt_1  )\cdot \nn\right)\nn  \quad \mbox{on} \quad \Sigma
\end{equation*}

Hence, according to \eqref{sv1b}, \eqref{divvbj}, and to \eqref{sv+j} and \eqref{divsv+j} when $j=1$, and according to  \eqref{dnsv1t}, \eqref{sv1n}, \eqref{fv0-} and \eqref{fv1n},  $(\Vmb_{1},\Vplus_{1})$, and $(\pmob_{1},\Pplus_{1})$ satisfy the transmission problem:
\begin{equation}
\label{SBBJ1}
 \left\{
   \begin{array}{lll}
    \kappa\Vmb_{1} + \nabla\pmob_{1}=0
 \quad&\mbox{in}\quad \Omega_{-}
\\[0.5ex]
\nabla\cdot\Vmb_{1}=0 \quad&\mbox{in}\quad \Omega_{-}\ 
\\[0.5ex]
 - \nabla \cdot \bsigma^+( \Vplus_{1} , \Pplus_{1})=0 
 \quad&\mbox{in}\quad \Omega_{+}
\\[0.5ex]
\nabla\cdot \Vplus_{1}= 0\quad&\mbox{in}\quad \Omega_{+}
 \\[0.5ex]
(\Vplus_{1} -\Vmb_{1})\cdot \nn =\nu_1^0           \quad &\mbox{on}\quad \Sigma
 \\[0.5ex]
  2\mu\bd(\Vplus_{1})\cdot\nn -\Pplus_{1}\nn =   -\pmob_{1}\nn +\frac{\beta}{2} \left((\Vplus_{1} +\Vmb_{1}) \cdot \nn\right)\nn  + \frac12 \sqrt{\kappa}\, \sw_0^0+\frac{\beta}{2} \nu_1^0 \nn
 \quad &\mbox{on} \quad \Sigma
\\[0.2ex]
\Vplus_{1}=0
  \quad &\mbox{on}\quad \Gamma \ .
   \end{array}
    \right.
\end{equation}

We apply Proposition \ref{profilbarre} with 
$$g^\pm=0, \quad l=\frac12 \sqrt{\kappa}\, \sw_0^0 + \frac \beta 2 \nu_1^0 \,\nn\in  \bH^{\bk- \frac 3 2} (\Sigma) , \quad \mbox{and}   \quad   h=\nu_1^0\in \H^{\sk-\frac 3 2}(\Sigma)\ ,
$$
 and $\bk\geqslant 2$. 
 We remark that $\nu_1^0$ satisfies the compatibility condition  $\int_{\Sigma} \nu_1^0 \dr \sigma=0$ since $\nu_1^0$ is the divergence of a tangential vector field. Indeed,  the divergence of $\sw_0^0$ on the surface $\Sigma$ is the divergence operator on $\Sigma$ applied to a tangent vector field on $\Sigma$;  hence by the Stokes formula,  the integral of $\nabla\cdot \sw_0^0$ vanishes since $\Sigma$ does not have boundary.

Hence we obtain the existence and uniqueness (up to an additive constant for the pressures) for the asymptotics $\Vmb_{1}$, $\Vplus_{1}$, $\pmob_{1}$ and $\Pplus_{1}$ which satisfy:
\begin{equation}
\Vplus_1\in \bH^{\sk}(\Omega_{+}), \quad \Vmb_1\in \bH^{\sk-1}(\Omega_{-}), \quad \Pplus_1\in \H^{\sk-1}(\Omega_{+})\quad\mbox{and} \quad \pmob_1\in \H^{\sk}(\Omega_{-}) \ .
\end{equation}

\subsubsection{Determination of $(\fv_{1})_\tau$}

 Using $\fp_{1}=0 $, then \eqref{fv1} writes now

\begin{equation}
\label{FV1}
-\dzz \fv_{1} + \kappa \fv_{1}   +\nabla d \ \dz \fp_{2}=-\sqrt{\kappa} \, ( 2 (\nabla d \cdot  \nabla)  \sw_0^0 + \Delta d \ \sw_0^0)\exp(-\sqrt{\kappa}\, \sz) \ , \quad \mbox{in} \quad \Omega_{-}\times \R^+\ .
\end{equation}

Taking the tangential components of the previous equation, we obtain that $\fv_{1}( x,\sz )_\tau$ solves 
\begin{equation}
\label{fv1t}
-\dzz(\fv_{1})_\tau + \kappa (\fv_{1})_\tau=-\sqrt{\kappa} \, ( 2 (\nabla d \cdot  \nabla)  \sw_0^0 + \Delta d \ \sw_0^0)_\tau\exp(-\sqrt{\kappa}\, \sz)\quad \mbox{in} \quad \Omega_{-}\times \R^+
\ ,
\end{equation}
with (see \eqref{svjtz=0}) : 
\begin{equation}
\label{CBfv1t}
 (\fv_{1})_\tau(x,0)
=
(\Vplus_{1} -\Vmb_{1})_\tau
 - \frac{2\mu}{\alpha} \left(\bd(\Vplus_{1})\cdot \nn\right)_\tau
 \quad \mbox{on}\quad \Sigma.
 \end{equation}

On the one hand, we introduce $\sw_1^0$ which is a tangential extension of $\dsp (\Vplus_{1} -\Vmb_{1})_\tau
 -\frac{2\mu}{\alpha} \left(\bd(\Vplus_{1})\cdot \nn\right)_\tau
$ in the domain  $\Omega_{-}$. We choose this extension $\sw_{1}^0$ such that it has a support in a tubular neighborhood of the interface $\Sigma$. Since $\Vplus_1\in \bH^{\sk}(\Omega_{+})$ and $\Vmb_1\in \bH^{\sk-1}(\Omega_{-})$, we can take an extension satisfying
 $$
 \sw_1^0\in \bH^{\sk-1}(\Omega_{-})\ .
 $$
  On the other hand, we  denote
  $$\sw_1^1=-\frac{1}{2}\left( 2\nabla d\cdot  \nabla \sw^0_0  + \Delta d\  \sw^0_0  \right)_\tau  \exp( -\sqrt {\kappa}\sz ).$$
 We remark that $\sw_1^1\in  \bH^{\sk-1}(\Omega_{-})$ since $\sw_0^0\in  \bH^{\sk}(\Omega_{-})$.
 
By solving (\ref{fv1t}) with the boundary condition (\ref{CBfv1t}) and we obtain that
\begin{equation}
\label{Efv1t}
(\fv_1)_\tau(x,\sz)= \left(\sw_1^0(x)+\sz\,\sw_1^1(x)\right) \exp ( -\sqrt {\kappa}\sz )
\quad \mbox{with} \quad \sw_1^0 \ \ , \sw_1^1 \quad\mbox{in} \quad \bH^{\sk-1}(\Omega_{-}) .
\end{equation}


\subsection{Characterization of the asymptotics of order $j \geq 2$.}

First let us write the equations satisfied by the asymptotics at order $j$.

At order $2$, from  the normal components of \eqref{fv1}, using \eqref{fv0-}, we have 
$$(-\dzz \Vmt_1 + \kappa \Vmt_1)\cdot \nn -\dz \fp_{2}= -\sqrt{\kappa} \, ( 2 (\nabla d \cdot  \nabla)  \sw_0^0 + \Delta d \ \sw_0^0)\cdot \nn \, \exp(-\sqrt{\kappa}\, \sz).
$$
Hence, according to \eqref{fv1n}, $ \fp_{2}$ satisfies 
$$
\dz \fp_{2}   =    \sqrt{\kappa} \, ( 2 (\nabla d \cdot  \nabla)  \sw_0^0) \cdot \nn  \exp(-\sqrt{\kappa}\, \sz)  \ , \quad \mbox{in} \quad \Omega_{-}\times \R^+\ .
$$
We infer 
\begin{equation}
\label{fp2}
 \fp_{2}=\sq_2^0  \exp(-\sqrt{\kappa}\, \sz)\quad \mbox{with}\quad \sq_2^0=             - ( 2 (\nabla d \cdot  \nabla)  \sw_0^0) \cdot \nn \in \bH^{\sk-1}(\Omega_{-}).
\end{equation}

At order $j\geq 3$, from the normal part of \eqref{svjt}, replacing $j$ by $j-1$,  we obtain that 
$$\dz \fp_{j}=\left(-\dzz \Vmt_{j-1}+\kappa \Vmt_{j-1}+\nabla\fp_{j-1}-2\nabla d\cdot \nabla \dz \Vmt_{j-2}-\Delta d \dz\Vmt_{j-2}-\Delta \Vmt_{j-3}\right)\cdot \nn,$$ 
so that
\begin{multline}
\label{ptmj}
\fp_{j}(x,z)=
\\
-\int_\sz^{+\infty}\left(-\dzz \Vmt_{j-1}+\kappa \Vmt_{j-1}+\nabla\fp_{j-1}-2\nabla d\cdot \nabla \dz \Vmt_{j-2}-\Delta d \dz\Vmt_{j-2}-\Delta \Vmt_{j-3}\right)\cdot \nn\, d\sz.
\end{multline}

From \eqref{divfvj}, there holds
$$ \dz \Vmt_j \cdot \nn = \nabla \cdot \Vmt_{j-1}\quad \mbox{in} \quad \Omega_{-}\times \R^+\ ,$$
hence
\begin{equation}
\label{defVmtjn}
\Vmt_j \cdot \nn(x,\sz)=-\int_\sz^{+\infty} \nabla \cdot \Vmt_{j-1} (.\ ,\xi)\,d\xi \ .
\end{equation}

From \eqref{svjb}, \eqref{sv+j}, \eqref{divvbj}, \eqref{divsv+j}, \eqref{dnsvjt}, \eqref{v+j} and \eqref{svjn}, we obtain the following system characterizing $\Vplus_j$, $\Pplus_j$, $\Vmb_j$ and $\pmob_j$:
\begin{equation}
\label{Vjbarre}
\left\{\begin{array}{ll}
\dsp  \kappa \Vmb_{j}  +\nabla\pmob_{j}  = \Delta\Vmb_{{j-2}} \quad &\mbox{in} \quad \Omega_{-}\ 
\\[0.5ex]
\nabla\cdot\Vmb_{j}=0 \quad&\mbox{in}\quad \Omega_{-}\ 
\\[0.5ex]
 - \nabla \cdot \bsigma^+(\Vplus_{j}, \Pplus_{j})=0 
 \quad&\mbox{in}\quad \Omega_{+}
\\[0.5ex]
\nabla\cdot \Vplus_{j}= 0\quad&\mbox{in}\quad \Omega_{+}
    \\[0.5ex]
(\Vplus_{j} -\Vmb_{j})\cdot \nn =h_j           \quad &\mbox{on}\quad \Sigma
 \\[0.5ex]
 2\mu\bd(\Vplus_{j})\cdot\nn - \Pplus_{j}\nn = -\pmob_{j}\nn +\frac{\beta}{2} \left((\Vplus_{j} +\Vmb_{j}   ) \cdot \nn \right)\nn +l_j \quad &\mbox{on} \quad \Sigma
\\[0.2ex]
\Vplus_{j}=0
  \quad &\mbox{on}\quad \Gamma \ ,
   \end{array}
    \right.
\end{equation}
where the data are given by:
$$
\begin{array}{l}
\dsp l_j=-\left(\dz \Vmt_{j-1}(\cdot,0)+ (\dz\Vmt_{j-1}\cdot \nn)\nn \right)+2( \bd(\Vmb_{j-2})\cdot \nn+\bd( \Vmt_{j-2})(\cdot,0)\cdot \nn)+\left(\frac{\beta}{2}\Vmt_j(\cdot , 0)\cdot \nn -\fp_j(\cdot , 0)\right)\nn,
\\[1ex]
\dsp h_2=\Vmt_2\cdot \nn(\cdot,0) +  2 \mu\left(\bd(\Vplus_{0})\cdot\nn\right)\cdot \nn -\Pplus_{0} -\frac \beta 4 \left (\Vplus_{0} +\Vmb_{0}+\Vmt_{0}\right)\cdot \nn -(\sh+ \frac 1 2 \sl)\cdot \nn,
\\[1ex]
\dsp h_j= \Vmt_j\cdot \nn(\cdot,0) +2\mu\left(\bd(\Vplus_{j-2})\cdot\nn\right)\cdot\nn -\Pplus_{j-2}-\fp_{j-2} -\frac{ \beta}{4}\left (\Vplus_{j-2} +\Vmb_{j-2}+ \Vmt_{j-2}\right)\cdot \nn \quad \mbox{for}\quad
j\geq 3 .
\end{array}
$$

Taking the tangential part of  \eqref{svjt} and writing \eqref{svjtz=0}, we obtain that $(\Vmt_j)_\tau$ satisfies
\begin{equation}
\label{detsvjt}
\left\{
\begin{array}{l}
-\dzz(\Vmt_{j})_\tau +\kappa (\Vmt_{j})_\tau= \left( -\nabla\fp_j + 2(\nabla d \cdot \nabla) \dz\Vmt_{j-1} + \Delta d \dz \Vmt_{j-1} + \Delta \Vmt_{j-2}\right)_\tau\quad \mbox{ for }\sz\geq 0,
\\\\[1ex]
\dsp   (\Vmt_{j})_\tau(\cdot,0)= (\Vplus_j-\Vmb_j)_\tau - \frac{2\mu}{\alpha} (\bd (\Vplus_j)\cdot \nn )_\tau.
  \end{array}
\right.
\end{equation}

\subsection{Existence and regularity of the asymptotics} 

\begin{prop}
\label{existence-profils}
Let $\bk\geqslant 1$. We assume that the data in \eqref{Das1}-\eqref{Das9} satisfy:
\begin{equation*}
\sg^-\in \bH^\bk(\Omega_{-}),\quad   \sg^+\in \bH^{\bk-1}(\Omega_{+}),\quad  \sl\in \bH^{\bk-\frac{1}{2}}(\Sigma)\quad\mbox{and}\quad\sh\in \bH^{\bk-\frac{1}{2}}(\Sigma).
\end{equation*}
Then, for all $j\in \{0, \ldots, \sk-1\}$, there exists asymptotics $\Vplus_j$, $\Vmb_j$, $\Vmt_j$, $\Pplus_j$, $\pmob_j$ and $\fp_j$, satisfying 
\begin{itemize}
\item
\eqref{Efp0},  \eqref{SBBJ0} and \eqref{fv0-}, when $j=0$,
\item
\eqref{fp1}, \eqref{fv1n}, \eqref{SBBJ1}, and \eqref{Efv1t}, when $j=1$, 
\item
\eqref{fp2}, \eqref{ptmj}, \eqref{defVmtjn}, \eqref{Vjbarre}
and \eqref{detsvjt}, when $j\geqslant 2$. 
\end{itemize}
There holds 
$$
\Vplus_j\in \bH^{\sk-j+1}(\Omega_{+}), \quad \Vmb_j\in \bH^{\sk-j}(\Omega_{-}), \quad \Pplus_j\in \H^{\sk-j}(\Omega_{+})\quad\mbox{and} \quad \pmob_j\in \H^{\sk-j+1}(\Omega_{-}).
$$
In addition, the boundary layer terms $\Vmt_j$ are on the following form: $\Vmt_0\cdot \nn =0$ and
$$\begin{array}{l}
\dsp \Vmt_j\cdot \nn (x,\sz)=\sum_{l=0}^{j-1} \nu_j^l(x)\sz^l  \exp(-\sqrt \kappa\, \sz)\quad \mbox{ with }\quad \nu_j^l\in \H^{\sk-j}(\Omega_{-}),
\\[1ex]
\dsp (\Vmt_j)_\tau(x,\sz)=\sum_{l=0}^{j} \sw_j^l(x)\sz^l  \exp(-\sqrt \kappa\, \sz)\quad \mbox{ with }\quad \sw_j^l\in \bH^{\sk-j}(\Omega_{-}).
\end{array}$$
Concerning the boundary layer terms for the pressure, $\fp_0=\fp_1=0$ and for $j\geq 2$, 
$$ \fp_{j}(x,\sz)=\sum_{l=0}^{j-2} \sq_{j}^l(x)\sz^l  \exp(-\sqrt \kappa\, \sz)\quad \mbox{ with }\quad \sq_j^l\in \bH^{\sk-j+1}(\Omega_{-})$$
\end{prop}

\begin{proof}
From Sections \ref{sec-as0} and \ref{sec-as1}, the property is true for $j=0$ and $j=1$. 

Let us assume that the property is true up to order $j-1$, with $j\leqslant \sk-1$. We claim without proof the following lemma:

\begin{lem}
\label{L1}
 For all $l\in \N$,
 $$\int_\sz^{+\infty}\xi^l\exp(-\sqrt {\kappa}\,\xi)d \,\xi=\left(\sum_{j=0}^l \frac{l!}{j!}\frac {1}{\kappa^{\frac{l-j+1}{2}}}\sz^j\right)\exp(-\sqrt \kappa \, \sz).$$ 
\end{lem}

\vspace{2mm}

{\it First step: construction of $\fp_j$.}

\vspace{2mm}

Concerning $\fp_j$, if $j=2$, we have Property \eqref{fp2}. If $j>2$, then from \eqref{ptmj} and with the induction property at order $j-1$ and $j-2$, $\dz \fp_j$ writes:
\begin{equation}
\label{Efpj}
\dz \fp_j(x,\sz)=\left( \sum_{l=0}^{j-2} \psi_l(x)\sz^l\right)\exp(-\sqrt \kappa \, \sz)
\end{equation}
where the terms $\psi_l$ are linear combinations of the terms $\nu_{j-1}^s$, $\nabla \sq^s_{j-1}$, $\nu_{j-2}^s$, $\sw_{j-2}^s$, $\nabla \nu_{j-2}^s$, $\nabla \sw_{j-2}^s$, $\Delta \nu_{j-3}^s$, and $\Delta \sw_{j-3}^s$. In particular, all these terms belong to the space $\bH^{k-j+1}(\Omega_{-})$. Hence, by using the previous lemma, we obtain by \eqref{Efpj} that $\fp_j$ writes:
$$ \fp_{j}(x,\sz)=\sum_{l=0}^{j-2} \sq_{j}^l(x)\sz^l  \exp(-\sqrt \kappa\, \sz)\quad \mbox{ with }\quad \sq_j^l\in \bH^{\sk-j+1}(\Omega_{-})\ .$$ 

\vspace{2mm}

{\it Second step: construction of $\Vmt_j\cdot \nn$.}

\vspace{2mm}

From \eqref{defVmtjn}, with the induction property at order $j-1$, we remark that $\dz \Vmt_j\cdot \nn$ writes
$$\dz \Vmt_j\cdot \nn=\left(\sum_{l=0}^{j-1}\phi_l(x)\sz^l\right)\exp(-\sqrt \kappa \, \sz)$$
where the terms $\phi_l$ are linear combinations of the $\nabla \cdot \nu_{j-1}^s$ and the $\nabla \cdot \sw_{j-1}^s$, so that the $\phi_i$ are in $\H ^{k-j}(\Omega_{-})$. Using the Lemma \ref{L1}, we obtain that $\Vmt_j\cdot \nn$ writes:
$$ \Vmt_j\cdot \nn (x,\sz)=\sum_{l=0}^{j-1} \nu_j^l(x)\sz^l  \exp(-\sqrt \kappa\, \sz)\quad \mbox{ with }\quad \nu_j^l\in \H^{\sk-j}(\Omega_{-}).$$

\vspace{2mm}

{\it Third step: construction of the terms $\Vplus_j$, $\Pplus_j$, $\Vmb_j$ and $\pmob_j$.}

\vspace{2mm}

The terms $\Vplus_j$, $\Pplus_j$, $\Vmb_j$ and $\pmob_j$ satisfy \eqref{Vjbarre}, and the data satisfy
$$\Delta \Vmb_{j-2}\in \bH^{\sk-j}(\Omega_{-}),\quad
l_j\in \bH^{\sk-j-\frac{1}{2}}(\Sigma), \quad h_j\in \H^{\sk-j-\frac{1}{2}}(\Sigma).$$
We remark that we can add an arbitrary constant to $\Pplus_{j-2}$ in order to obtain the compatibility condition:
$$\int_\Sigma h_j d\, \sigma=0.$$
Hence, we can apply Proposition \ref{profilbarre}  with $k=  \sk-j\geqslant 1$, $g^-=\Delta\Vmb_{j-2}$, $g^+=0$, $l=l_j$, $h=h_j$.  Therefore we obtain the existence and the uniqueness (up to an additive constant for the pressures) of the solution of \eqref{Vjbarre} and the solution satisfies:

$$
\Vplus_j\in \bH^{\sk-j+1}(\Omega_{+}), \quad \Vmb_j\in \bH^{\sk-j}(\Omega_{-}), \quad \Pplus_j\in \H^{\sk-j}(\Omega_{+})\quad\mbox{and} \quad \pmob_j\in \H^{\sk-j+1}(\Omega_{-}).
$$

\vspace{2mm}

{\it Fourth step: construction of $(\Vmt_j)_\tau$.}

\vspace{2mm}

On the interface $\Sigma$, from the previous results, $\dsp \alpha(\Vplus_j-\Vmb_j)_\tau -2 \mu (\partial_\nn \Vplus_j)_\tau \in \bH^{\sk-j-\frac 1 2}(\Sigma)$. We introduce $\sw_j^0\in \bH^{\sk-j}(\Omega_{+})$ a tangential extension of this boundary data. We choose this extension $\sw_{j}^0$ such that it has a support in a tubular neighborhood of the interface $\Sigma$.

\vspace{1mm}

We have now the following lemma:

\begin{lem} 
\label{lem-2}
Let $K\geqslant 0$. The solution of the ODE
$$\left\{
\begin{array}{l}
\dsp -\dzz f + \kappa f = \left(\sum_{i=0}^K \gamma_i \sz^i\right)\exp(-\sqrt \kappa \sz)
\quad \mbox{for all} \quad \sz\geqslant 0\ 
\\
f(0)=f_0\\ \\
f(\sz)\ds 0 \quad \mbox{when} \quad \sz\ds +\infty,
\end{array}
\right.$$
is $\dsp f(z)=\left( f_0 + \sum_{i=0}^K\beta_i \sz^{i+1}\right)\exp(-\sqrt \kappa \sz)$. Here,  
$$\left(\begin{array}{c}
\gamma_0\\
\cdot\\
\cdot\\
\gamma_K
\end{array}\right)=
M_K
\left(\begin{array}{c}
\beta_0\\
\cdot\\
\cdot\\
\beta_K
\end{array}\right)
$$
where $M_K$ is a $(K+1)\times(K+1)$ nonsingular matrix with entries $m_{ij}$ such that
$m_{ii}=2i\sqrt \kappa$, $m_{i, i+1}=-i(i+1)$, and with vanishing other entries.
\end{lem}

From \eqref{detsvjt}, using the induction hypothesis on $\fp_{j}$, $\Vmt_{j-1}$ and on $\Vmt_{j-2}$, $(\Vmt_j)_\tau$ satisfies the system:

$$
\left\{
\begin{array}{l}
-\dzz (\Vmt_j)_\tau(x,\sz)+\kappa (\Vmt_j)_\tau(x,\sz)=\di\left(\sum_{i=0}^{j-1} \gamma_i (x)\sz^i\right)\exp(-\sqrt \kappa \sz), 
\quad \mbox{for}\quad  x\in \Omega_{-},\; \sz\geq 0,\\
\\
(\Vmt_j)_\tau(x,0)=\sw_j^0(x) \quad \mbox{for}\quad x \in \Omega_{-},
\end{array}\right.$$
where $\gamma_i\in \bH^{\sk-j}(\Omega_{-})$ and $\sw_j^0\in \bH^{\sk-j}(\Omega_{-})$. Hence, by applying Lemma \ref{lem-2}, we obtain that 
$$\dsp (\Vmt_j)_\tau(x,\sz)=\sum_{l=0}^{j} \sw_j^l(x)\sz^l  \exp(-\sqrt \kappa\, \sz)\quad \mbox{ with }\quad \sw_j^l\in \bH^{\sk-j}(\Omega_{-}).$$
This concludes the proof of the property at order $j$.

 \end{proof}



\section{Regularity proof for the asymptotics}
\label{secregprof}

Following the previous section, we want to solve a collection of elementary problems satisfying :   Find $\sv=(\Vmb,\Vplus)$, and $\sp=(\pmob,\Pplus)$ such that 
\begin{equation}
\label{SBBJj}
 \left\{
   \begin{array}{lll}
    \kappa\Vmb + \nabla\pmob=g^-
 \quad&\mbox{in}\quad \Omega_{-}
\\[0.5ex]
\nabla\cdot\Vmb=0 \quad&\mbox{in}\quad \Omega_{-}\ 
\\[0.5ex]
-\nabla \cdot \bsigma^+(\Vplus,\Pplus)=g^+
 \quad&\mbox{in}\quad \Omega_{+}
\\[0.5ex]
\nabla\cdot \Vplus= 0\quad&\mbox{in}\quad \Omega_{+}
\\[0.5ex]
\bsigma^+(\Vplus, \Pplus)\cdot \nn  = l-\pmob \nn +\frac \beta 2 ((\Vplus+\Vmb)\cdot \nn) \nn
  \quad &\mbox{on} \quad \Sigma
    \\[0.5ex]
(\Vplus -\Vmb)\cdot \nn= h
\quad &\mbox{on}\quad \Sigma
\\[0.2ex]
\Vplus=0
  \quad &\mbox{on}\quad \Gamma \ ,
   \end{array}
    \right.
\end{equation}
where $\bsigma^+ (\Vplus,\Pplus)=2\mu\bd(\Vplus)- \Pplus\nn$, associated with the data $g, h, l$. We remark that because of the divergence free condition, we need the compatibility condition $\dsp \int_\Sigma h d\sigma=0$.

\begin{rem}
The first asymptotic $\sv_{0}=(\Vplus_{0},\Vmb_{0})$ satisfies the  problem \eqref{SBBJj}  with $g^\pm=\sg^\pm$, $ l=\sl$ and $h=0$.  The term $\sv_{1}=(\Vplus_{1},\Vmb_{1})$ satisfies the problem \eqref{SBBJj}  with $g\pm=0$, $h=\nu_{1}^0$, $ l_\tau=\sqrt{\kappa}\ \sw_{0}^0$ and  $l\cdot \nn=\frac{\beta}{2} \nu_{1}^0$.  
\end{rem}

We first address this problem when $h=0$.

\subsection{Elementary problem without jump for the normal components}

We  consider the following problem:
\begin{equation}
\label{SBBJk}
 \left\{
   \begin{array}{lll}
    \kappa \Vmb+ \nabla \pmob=g^{-}
 \quad&\mbox{in}\quad \Omega_{-}
\\[0.5ex]
\nabla\cdot \Vmb=0 \quad&\mbox{in}\quad \Omega_{-}\ 
\\[0.5ex]
 -\nabla \cdot \bsigma^+(\Vplus,\Pplus)=g^+
 \quad&\mbox{in}\quad \Omega_{+}
\\[0.5ex]
\nabla\cdot \Vplus= 0\quad&\mbox{in}\quad \Omega_{+}
\\[0.5ex]
 \bsigma^+(\Vplus, \Pplus)\cdot \nn  = l-\pmob \nn +\frac \beta 2 ((\Vplus+\Vmb)\cdot \nn) \nn
  \quad &\mbox{on} \quad \Sigma
    \\[0.5ex]
(\Vplus -\Vmb)\cdot \nn= 0 \quad &\mbox{on}\quad \Sigma
 \\[0.2ex]
\Vplus=0
  \quad &\mbox{on}\quad \Gamma \ .
   \end{array}
    \right.
\end{equation}

Then, we introduce a variational problem for $\sv=(\Vmb,\Vplus)$ associated with the problem \eqref{SBBJk} in the space
$$
W=\{ \su \in \bL^2 (\Omega) \ | \, \nabla \su^+ \in \bL^2(\Omega_{+ }) \, , \quad \Div \su =0  \quad  \mbox{in}\quad \Omega , \quad \su^+ =0 \quad  \mbox{on}  \quad\Gamma
\} \ .
$$ 
Such a  variational formulation writes : Find $\sv=(\Vmb,\Vplus)  \in W$ such that
\begin{equation}
   \forall \su \in  W, \quad 
   a(\sv,\su) = b(\su) ,
\label{E1}
\end{equation}
where
$$  a (\sv,\su) :=\mu\int_{\Omega_{+}} \bd (\Vplus) :\bd(\su^+) \,\dr\xx
  + \kappa \int_{\Omega_{-}} \Vmb \cdot \su^-\, \dr\xx 
  +\beta  \!\int_{\Sigma}  
   (\sv \cdot \nn)( \su \cdot \nn)  
  \, \dr s \ ,
$$
and
$$
 b(\su)
     =   \int_{\Omega}  g \cdot \su  \,\dr\xx 
     - \!\int_{\Sigma}  l\cdot \su^+
  \, \dr s \   .
$$  
Endowed with the norm 
$$
\| \sv \|_{W}= \| \Vplus \|_{1,\Omega_{+}}+ \| \Vmb \|_{0,\Omega_{-}}+  \| \sv\cdot\nn \|_{0,\Sigma} \ ,
$$
 the functional space $W$ is a Hilbert space since $W$ is a closed subspace of the Hilbert space $\{ \su \in \bL^2 (\Omega) \ | \, \nabla \su^+ \in \bL^2(\Omega_{+ }) \, , \quad \su^+ =0 \quad  \mbox{on}  \quad\Gamma\} $.

As a consequence of both the Poincar\'e inequality in $ \bH_{0,\Gamma}^1(\Omega_{+})$ and the Lax-Milgram lemma, we infer the well-posedness of problem \eqref{E1} :

\begin{prop}
\label{laxmil}
For given data $g\in\bL^2(\Omega)$  and $l\in  \bL^2(\Sigma)$,  the  problem \eqref{E1} is well posed in $W$. 
\end{prop}

If we assume in addition that $\rot g^-\in \bL^2(\Omega_{-})$, then the solution $\sv$ of the problem \eqref{E1}  belongs to the space $V$ (we remind that $V$ is introduced in the beginning of Section \ref{UE}). 
\begin{prop}
For given data $g\in\bL^2(\Omega)$ with $\rot g^-\in \bL^2(\Omega_{-})$,  and $l\in  \bL^2(\Sigma)$,  the solution $\sv$ of the  problem \eqref{E1}  belongs to the space $V$. 
\end{prop}

\begin{proof}
Taking the $\rot$ in the first equation in \eqref{SBBJk}, $ \kappa \rot\Vmb =\rot g^{-}\in \bL^2( \Omega_{-})$. Moreover $\nabla\cdot \Vmb=0$ and $\Vmb\cdot\nn=\Vplus\cdot\nn \in \H^{\frac12}(\Sigma)$ since $\Vplus$ belongs to  $\bH^1(\Omega_{+})$. Since $\Omega_{-}$ is a smooth domain, we infer $\Vmb \in \bH^1(\Omega_{-})$.
\end{proof}

The next proposition ensures a regularity result in Sobolev spaces for the solutions of problem \eqref{SBBJk}. It is the main result of this section.

\begin{prop}
\label{prop-reg1}
Let $k\geq 1$. We assume that $g^+\in \bH^{k-1}(\Omega_{+})$, $g^-\in \bH^k(\Omega_{-})$, and $l\in H^{k-\frac 1 2}(\Sigma)$. Then the solution of the problem \eqref{SBBJk} satisfies
$\Vplus\in \bH^{k+1}(\Omega_{+})$, $\Vmb\in \bH^{k}(\Omega_{-})$, $\Pplus\in \H^{k}(\Omega_{+})$ and $\pmob\in \H^{k+1}(\Omega_{-})$.
\end{prop}

\begin{proof} 
We remark that if we know $\Vmb\cdot \nn$ on $\Sigma$, then $\pmob$ and $\Vmb$ are completely determined. Indeed, taking the divergence of the first equation in \eqref{SBBJk} we obtain that
$$\Delta \pmob=\nabla \cdot g^-\ .$$
In addition, taking the scalar product of the same equation with $\nn$, we obtain that
$$\kappa \Vmb\cdot \nn+\frac{\partial \pmob}{\partial \nn}=g^-\cdot \nn.$$
Thus, $\pmob$ satisfies:
$$
\left\{
\begin{array}{l}
  \dsp \Delta \pmob=\nabla \cdot g^-\quad \mbox{in} \quad \Omega_{-}\\\\
\dsp \frac{\partial \pmob}{\partial \nn}=g^-\cdot \nn-\kappa \Vmb\cdot \nn\quad
\mbox{on} \quad \Sigma
\end{array}
\right.
$$
With an additional condition on the mean of $\pmob$, this characterizes completely $\pmob$.

We fix $j\in \{1, \cdot ,k\}$. We introduce the following Dirichlet to Neumann operator $T:\H^{j-\frac{1}{2}}(\Sigma)\longrightarrow \H^{j+\frac{1}{2}}(\Sigma)$ in the following way:
for $\varphi \in \H^{j-\frac{1}{2}}(\Sigma)$ we solve 
\begin{equation}
\label{pmbDN}
\left\{
\begin{array}{l}
  \dsp \Delta p=\nabla \cdot g^-\quad \mbox{in} \quad \Omega_{-}\\\\
\dsp \frac{\partial p}{\partial \nn}=g^-\cdot \nn-\kappa \varphi
\mbox{ on } \quad \Sigma\\
\\
\int_{\Omega_{-}}p \ \dr \xx=0
\end{array}
\right.
\end{equation}
and we denote by 
$T(\varphi)$ the trace of the obtained $p$ on $\Sigma$.

We remark that since $g^-\in \bH^k(\Omega_{-})$, if $\varphi \in \H^{j-\frac{1}{2}}(\Sigma)$ then by classical elliptic regularity results, $p\in \H^{j+1}(\Omega_{-})$ so that there exists a constant $C$ independent on $\varphi$ and $g^-$ such that
\begin{equation}
\label{esti-DtoN}
\Vert p \Vert_{\H^{j+1}(\Omega_{-})}+
\Vert T(\varphi)\Vert_{\H^{j+\frac{1}{2}}(\Sigma)}\leq C(\Vert g^-\Vert_{ \bH^k(\Omega_{-})} + \Vert \varphi\Vert_{\H^{j-\frac{1}{2}}(\Sigma)}).
\end{equation}

In addition, we have $ v^-=\dsp \frac{1}{\kappa}(g^- - \nabla p^-)$, and we infer
$$\Vert v^-\Vert_{\bH^{j}(\Omega_{-})}\leq C (\Vert g^-\Vert_{ \bH^k(\Omega_{-})} + \Vert \varphi\Vert_{\H^{j-\frac{1}{2}}(\Sigma)}).$$

Now on the domain $\Omega_{+}$, we rewrite the boundary conditions: we remark first that $\Vplus\cdot \nn=\Vmb\cdot \nn$ on $\Sigma$. In addition, from the equations on $\Omega_{-}$, $\pmob=T(\Vmb\cdot \nn)$. Therefore, $(\Vplus, \Pplus)$ satisfies the problem:

\begin{equation}
\label{VplusDtoN}
\left\{\begin{array}{l}
-\nabla \cdot \bsigma^+(\Vplus,\Pplus)=g^+
\quad \mbox{in} \quad \Omega_{+}\\\\
\nabla\cdot \Vplus=0 \quad \mbox{in} \quad \Omega_{+}\\\\
\dsp \bsigma^+(\Vplus,\Pplus)\cdot \nn = l + \beta (\Vplus\cdot \nn)\nn - T(\Vplus \cdot \nn)\nn  \quad \mbox{on} \quad\Sigma\\
\\
\Vplus = 0  \quad \mbox{on} \quad \Gamma \ .
\end{array}
\right.
\end{equation}

Hereafter we use the next proposition which ensures a well-posedness result together with elliptic regularity result for the Stokes operator with mixed boundary conditions (namely with a stress boundary condition on $\Sigma$ and a Dirichlet boundary condition on $\Gamma$):
\begin{prop}
\label{prop-BFab}
Let $g\in \bL^{2}(\Omega_{+})$ and $\gamma \in \bH ^{-\frac 1 2}(\Sigma)$. Then the problem  
\begin{equation}
\label{Eprop-BFab}
\left\{\begin{array}{l}
-\nabla \cdot \bsigma^+(w,p)=g
\quad \mbox{in} \quad \Omega_{+}\\
\\
\nabla \cdot w=0\quad \mbox{in} \quad \Omega_{+}
\\ \\
w=0 \quad \mbox{in} \quad \Gamma
\\ \\
\dsp \bsigma^+(w,p)\cdot \nn =\gamma 
\quad \mbox{in} \quad \Sigma
\end{array}\right.
\end{equation}
has a unique solution $(w,p)$ in the space $\bH_{0,\Gamma}^1(\Omega_{+})\times \L^2(\Omega_{+})$. 

Let $s\geqslant 1$.  We assume that there exists a solution $(w,p)$ of the problem \eqref{Eprop-BFab} in the space $\bH^{s}(\Omega_{+})\times\H^{s-1}(\Omega_{+})$. If $g\in \bH^{s-1}(\Omega_{+})$ and $\gamma \in \bH ^{s-\frac 1 2}(\Sigma)$, then $(w,p)\in \bH^{s+1}(\Omega_{+})\times\H^{s}(\Omega_{+})$ and there exists a constant $C>0$ such that
$$\Vert w\Vert_{\bH^{s+1}(\Omega_{+})} + \Vert p \Vert_{\H^{s}(\Omega_{+})}\leq C\left( \Vert g\Vert_{\bH^{s-1}(\Omega_{+})} + \Vert \gamma \Vert_{\bH^{s-\frac 1 2}(\Sigma)}\right)\ .$$
\end{prop}
This proposition is a consequence of elliptic regularity result for the Stokes operator with normal stress boundary conditions (see \cite[Th. III.5.7]{BF05} on page 192).

Using Prop. \ref{prop-BFab}, let us prove now by induction on $j$ that for all $j\in\{1, \ldots, k\}$, we have the property ${\mathcal P}(j)$:
$${\mathcal P}(j):\quad \Vplus\in \bH^{j+1}(\Omega_{+}), \quad \Pplus\in \H^j(\Omega_{+}).$$

\vspace{2mm}

{\bf Proof of ${\mathcal P}(1)$:} we already know that $\Vplus\in \bH^1(\Omega_{+})$ using   Proposition \ref{prop-reg1}. Hence,  $\Vplus\cdot \nn\in \H^{\frac 1 2}(\Sigma)$. Thus, by \eqref{esti-DtoN}, we obtain that the right-hand-side of the third equation in \eqref{VplusDtoN} is in $ \bH^{\frac 1 2}(\Sigma)$ and thus, by Proposition \ref{prop-BFab}, we obtain that
$$\Vplus\in \bH^2(\Omega_{+})\quad\mbox{ and }\quad \Pplus\in \H^1(\Omega_{+}).$$

\vspace{2mm}
{\bf Induction:} we assume that $j<k$ and that ${\mathcal P}(j)$ is satisfied, {\it i.e.}   $\Vplus\in \bH^{j+1}(\Omega_{+})$. Hence $\Vplus\cdot \nn \in \H^{j+\frac 1 2}(\Sigma)$   . Thus, by \eqref{esti-DtoN}, we obtain that the right-hand-side of the third equation in \eqref{VplusDtoN} is in $ \bH^{j+\frac 1 2}(\Sigma)$ and thus, by Proposition \ref{prop-BFab}, we obtain that
$$\Vplus\in \bH^{j+2}(\Omega_{+})\quad\mbox{ and }\quad \Pplus\in \H^{j+1}(\Omega_{+}).$$

This complete the proof of Proposition \ref{prop-reg1}.

\end{proof}

\subsection{General existence and regularity result for the elementary problems}

We address now the elementary problem with non vanishing jump for the normal velocity \eqref{SBBJj}. 
We prove the following result:

\begin{prop}
\label{prop-elem1}
Let $k\geq 1$. We consider $g^-\in \bH^k(\Omega_{-})$, $g^+\in \bH^{k-1}(\Omega_{+})$, $l\in \bH^{k-\frac{1}{2}}(\Sigma)$ and $h\in \H^{k-\frac{1}{2}}(\Sigma)$. Assume that $h$ satisfies the compatibility condition 
$$\int_\Sigma h \,\dr \sigma=0.$$
Then the solution of problem \eqref{SBBJj} satisfies $\Vplus\in \bH^{k+1}(\Omega_{+})$, $\Vmb\in \bH^{k}(\Omega_{-})$, $\Pplus\in \H^{k}(\Omega_{+})$ and $\pmob\in \H^{k+1}(\Omega_{-})$.
\end{prop}

\begin{proof}
Let $w\in \bH^k(\Omega_{-})$ such that $\nabla \cdot w=0$ in $\Omega_{-}$ with $w\cdot \nn=- h$ on $\Sigma$. We denote $\sw=\Vmb-w$. Then $(\Vplus,\Pplus,\Vmb,\pmob)$ satisfies \eqref{SBBJj} if and only if $(\Vplus,\Pplus,\sw,\pmob)$ satisfies: 

\begin{equation}
\label{SBBJktildes}
 \left\{
   \begin{array}{lll}
    \kappa \sw+ \nabla \pmob=g^{-}-\kappa w
 \quad&\mbox{in}\quad \Omega_{-}
\\[0.5ex]
\nabla\cdot \sw=0 \quad&\mbox{in}\quad \Omega_{-}\ 
\\[0.5ex]
-\nabla \cdot \bsigma^+ ( \Vplus,\Pplus)=g^+
 \quad&\mbox{in}\quad \Omega_{+}
\\[0.5ex]
\nabla\cdot \Vplus= 0\quad&\mbox{in}\quad \Omega_{+}
\\[0.5ex]
\dsp \bsigma^+(\Vplus, \Pplus)\cdot\nn = -\pmob \nn +\frac \beta  2\left((\Vplus+\sw) \cdot \nn\right) \nn+  l - \frac \beta 2 h \, \nn
 \quad &\mbox{on} \quad \Sigma
    \\[0.5ex]
(\Vplus -\sw)\cdot \nn= 0\quad &\mbox{on}\quad \Sigma
\\[0.2ex]
\Vplus=0
  \quad &\mbox{on}\quad \Gamma \ .
   \end{array}
    \right.
\end{equation}

We remark that $g^--\kappa w\in \bH^k(\Omega_{-})$ and that $ l - \frac \beta 2 h \, \nn \in \bH^{k-\frac{1}{2}}(\Sigma)$ therefore, by Proposition \ref{prop-reg1}, we obtain that $\Vplus\in \bH^{k+1}(\Omega_{+})$, $\Pplus\in \H^{k}(\Omega_{+})$ and $\pmob\in \H^{k+1}(\Omega_{-})$. We obtain in addition that $\sw \in \bH^{k}(\Omega_{-})$ and since $w\in \bH^{k}(\Omega_{-})$, we obtain the same regularity for $\Vmb$.
\end{proof}


\section{Estimates of remainders}
\label{SERem}

We address the validation of the WKB expansion found before. We claim the following theorem:

\begin{thm}
\label{DAS-WKB}
Let $\bk\geq 5$. We assume that the data in \eqref{0SBBJ} satisfy hypothesis \eqref{regdata} that we recall here:
$$
\sg^-\in \bH^\bk(\Omega_{-}),\quad   \sg^+\in \bH^{\bk-1}(\Omega_{+}),\quad  \sl\in \bH^{\bk-\frac{1}{2}}(\Sigma)\quad\mbox{and}\quad\sh\in \bH^{\bk-\frac{1}{2}}(\Sigma).
$$
We fix $k=\bk -2$.
Let $\Vplus_j$, $\Vmb_j$, $\Vmt_j$, $\Pplus_j$, $\pmob_j$ and $\fp_j$,$j\in \{0, \ldots, \sk-1\}$ given by Proposition \ref{existence-profils}. We define  $(\sr_{k,\eps}^\pm, \sq_{k,\eps}^\pm)$  by removing to the solution $( \sv_{\eps},  \sp_{\eps})$ of \eqref{0SBBJ} the truncated expansion up to the order $\eps^{\frac k2}$ :
\begin{gather}
\label{Er+}
  \sr_{k,\eps}^+(x)= \vepsp(x) - \di\sum_{j=0}^{k}\eps^{\frac j2} \Vplus_j(x)    \, ,\\
   \label{Er-}
    \sr_{k,\eps}^-(x) = \vepsm(x) - \di\sum_{j=0}^{k}\eps^{\frac j2} \left(  \Vmb_{j}(x)  + \fv_j(x,\frac{d(x)}{\sqrt\eps})\right)    , \\
 \label{Eq+}
  \sq_{k,\eps}^+(x)=   \sp^+_{\eps}(x) - \di\sum_{j=0}^{k}\eps^{\frac j2}\sp^+_j(x)   \, ,\\
  \label{Eq-}
 \sq_{k,\eps}^-(x)=  \sp^-_{\eps}(x) - \di\sum_{j=0}^{k} \eps^{\frac j2} \left( \pmob_{j}(x)  +  \fp_j(x,\frac{d(x)}{\sqrt\eps})\right)\, .
\end{gather}

We have the following estimate : 
$$
 \eps \|\nabla{\sr}_{k,\eps}^{-} \|_{0,\Omega_{-}}^2
  +  \frac {\kappa}{4} \|{\sr}_{k,\eps}^{-} \|^2_{0,\Omega_{-}}  
  + \frac{ \mu C^2}{3} \|{\sr}_{k,\eps}^{+} \|_{1,\Omega_{+}}^2
\leqslant
C\eps^{ \frac{k-2}{2}}.$$

\end{thm}

\begin{proof}

By construction of the profiles, we derive  :
$$
 \left\{
   \begin{array}{lll}
- \eps \Delta \sr_{k,\eps}^{-} + \nabla \sq_{k,\eps}^-+\kappa\sr_{k,\eps}^{-}= \sg_ {k,\eps}  \quad&\mbox{in}\quad \Omega_{-}
\\[0.5ex]
 - \nabla \cdot \bsigma^+( \sr_{k,\eps}^+ , \sq_{k,\eps}^+)= 0
  \quad&\mbox{in}\quad \Omega_{+}
\\[0.5ex]
\nabla\cdot  \sr_{k,\eps}^- =\sf_{k, \eps} \quad &\mbox{in} \quad  \Omega_{-}
\\[0.5ex]
\nabla\cdot \sr_{k,\eps}^+=0 \quad&\mbox{in}\quad  \Omega_{+}
\\[0.5ex]
2\mu \bd(\sr_{k,\eps}^+)\cdot \nn -\sq_ {k,\eps}^+ \nn =2\eps \bd(\sr_{k,\eps}^-)\cdot \nn -\sq_ {k,\eps}^+\nn + \frac \beta 2 \left((\sr_{k,\eps}^++\sr_{k,\eps}^-
)\cdot \nn \right)\nn+\sl_ {k,\eps}
 \quad &\mbox{on} \quad \Sigma
  \\[0.5ex]
 \alpha (\sr_{k,\eps}^+ -\sr_{k,\eps}^-)_\tau= 2\mu \left(\bd( \sr_{k,\eps}^+)\cdot \nn \right)_\tau \quad &\mbox{on} \quad \Sigma
 \\[0.5ex]
 \frac{1}{\eps}( \sr_{k,\eps}^+ - \sr_{k,\eps}^-)\cdot \nn= 2\mu \left(\bd( \sr_{k,\eps}^+)\cdot \nn\right)\cdot \nn - \sq_{k,\eps}^+-\frac\beta 4 \left( \sr_{k,\eps}^+ + \sr_{k,\eps}^-\right)\cdot \nn + h_{k,\eps} \quad &\mbox{on} \quad \Sigma
\\[0.5ex]
\sr^+_ {k,\eps}=0
  \quad &\mbox{on}\quad \Gamma \ .
   \end{array}
    \right.
$$
with
$$
\begin{array}{l}
\dsp \sg_ {k,\eps}=\eps^{\frac{k+1}{2}}\left(\Delta \Vmt_{k-1} + 2(\nabla d\cdot \nabla)\dz\Vmt_{k} + \Delta d \dz \Vmt_{k} + \Delta \Vmt_{k-1}\right) +  \eps^{\frac{k +2}{2}} \left(
\Delta \Vmb_{k}+ \Delta \Vmt_{k}\right),\\
\\[1ex]
\dsp  \sf_{k, \eps}=    - \eps^{\frac{k }{2}} \nabla\cdot  \Vmt_{k}   ,
\\\\[1ex]
\sl_{k,\eps}= \eps^{\frac{k+1}{2}}\left(2 \bd(\Vmb_{k-1})\cdot \nn + 2\bd( \Vmt_{k-1})\cdot \nn - \dz\Vmt_{k}- (\dz \Vmt_k\cdot \nn)\nn\right)
 + \eps^{\frac{k +2}{2}} \left( \bd( \Vmb_{k})\cdot \nn  + \bd( \Vmt_{k})\cdot \nn \right),
\\\\[1ex]
h_{\eps,k}=  \eps^{\frac{k -1}{2}}\left( \bsigma^+( \Vplus_{k-1},\Pplus_{k-1}) \cdot \nn - \frac \beta 4 (\Vplus_{k-1} + \Vmb_{k-1} + \Vmt_{k-1})\right)\cdot \nn \\
\quad \quad + 
 \eps^{\frac{k }{2}}\left( \bsigma^+( \Vplus_{k},\Pplus_{k}) \cdot \nn - \frac \beta 4 (\Vplus_{k} + \Vmb_{k} + \Vmt_{k})\right)\cdot \nn.
\end{array}$$

Let us introduce $\psi_{k,\eps}:\Omega_-\longrightarrow \R^3$ such that
$$
\dsp \nabla \cdot \psi_{k,\eps}=\sf_{k,\eps} \quad  \mbox{in} \quad \Omega_{-}\ .
$$

From the expression of $\Vmt_k$ given by Proposition \ref{existence-profils}, since $\sz\mapsto \sz^l \exp (-\sqrt \kappa \, \sz)$ is bounded on $\R^+$, we obtain that
$$\Vert \sf_{k,\eps}\Vert_{\L^2(\Omega_-)}\leq C\eps ^{\frac{k}{2}} \max_l \{ \Vert \sw_k^l\Vert_{\H^1(\Omega_-)},  \Vert \nu_k^l\Vert_{\H^1(\Omega_-)}\}.$$

In addition, 
$$
\dsp \partial_i \sf_{k,\eps}= \dsp -\eps ^{\frac{k}{2}} (\nabla \cdot \partial_i \Vmt_k) (x,\frac{d(x)}{\sqrt\eps})-\eps ^{\frac{k-1}{2}} \partial_i d(x) (\nabla \cdot \dz \Vmt_k)(x,\frac{d(x)}{\sqrt\eps}).$$

Hence, there exists a constant $C>0$ such that 
$$\Vert \nabla \sf_{k,\eps}\Vert_{\L^2(\Omega_-)}\leq C\eps ^{\frac{k-1}{2}} \max_l \{ \Vert \sw_k^l\Vert_{\H^2(\Omega_-)},  \Vert \nu_k^l\Vert_{\H^2(\Omega_-)}\}.$$

Therefore we can assume that there exists $C>0$ such that for all $k$ and $\eps$,
\begin{equation}
\label{Epsikeps}
\Vert  \psi_{k,\eps} \Vert_{\bH^1(\Omega_-)}\leq C\eps ^{\frac{k}{2}}\max_l \{ \Vert \sw_k^l\Vert_{\H^1(\Omega_-)},  \Vert \nu_k^l\Vert_{\H^1(\Omega_-)}\}
\end{equation}
and 
$$\Vert  \psi_{k,\eps} \Vert_{\bH^2(\Omega_-)}\leq C\eps ^{\frac{k-1}{2}}\max_l \{ \Vert \sw_k^l\Vert_{\H^2(\Omega_-)},  \Vert \nu_k^l\Vert_{\H^2(\Omega_-)}\}.$$
Now we denote $\overline{\sr}_{k,\eps}^-=\sr_{k,\eps}^-- \psi_{k,\eps}$ and we obtain that $\overline{\sr}_{k,\eps}^-$, ${\sr}_{k,\eps}^+$ and $\sq_{k,\eps}^\pm$ satisfy:
\begin{equation}
\label{SBBJrem}
 \left\{
   \begin{array}{lll}
- \eps \Delta \overline{\sr}_{k,\eps}^{-} + \nabla \sq_{k,\eps}^-+\kappa \overline{\sr}_{k,\eps}^{-}= \overline \sg_ {k,\eps}   \quad&\mbox{in}\quad \Omega_{-}
\\[1ex]
 - \nabla \cdot \bsigma^+( \sr_{k,\eps}^+ , \sq_{k,\eps}^+)= 0
  \quad&\mbox{in}\quad \Omega_{+}
\\[1ex]
\nabla\cdot \overline{ \sr}_{k,\eps}^- =0  \quad &\mbox{in} \quad  \Omega_{-}
\\[1ex]
\nabla\cdot \sr_{k,\eps}^+=0 \quad&\mbox{in}\quad  \Omega_{+}
\\[1ex]
2\mu \bd(\sr_{k,\eps}^+)\cdot \nn -\sq_ {k,\eps}^+ \nn =2\eps \bd(\overline{\sr}_{k,\eps}^-)\cdot \nn -\sq_ {k,\eps}^+\nn + \frac \beta 2 \left((\sr_{k,\eps}^++\overline{\sr}_{k,\eps}^-
)\cdot \nn \right)\nn+\overline{\sl}_ {k,\eps}
 \\[1ex]
 \alpha (\sr_{k,\eps}^+ -\overline{\sr}_{k,\eps}^-)_\tau=2 \mu \left(\bd (\sr_{k,\eps}^+)\cdot \nn\right)_\tau
+\alpha \left( \psi_{k,\eps} \right)_\tau
  \quad &\mbox{on} \quad \Sigma
 \\[1ex]
\dsp  \frac{1}{\eps}( \sr_{k,\eps}^+ - \overline{\sr}_{k,\eps}^-)\cdot \nn= 2\mu  \left(\bd (\sr_{k,\eps}^+)\cdot \nn\right)\cdot \nn - \sq_{k,\eps}^+-\frac\beta 4 \left( \sr_{k,\eps}^+ + \overline{\sr}_{k,\eps}^-\right)\cdot \nn + \overline h_{k,\eps} 
  \quad &\mbox{on} \quad \Sigma
\\[1ex]
\sr^+_ {k,\eps}=0
  \quad &\mbox{on}\quad \Gamma \ ,
   \end{array}
    \right.
\end{equation}

where
$$
\begin{array}{l}
\dsp\overline \sg_ {k,\eps}=  \sg_ {k,\eps}  + \eps\Delta \psi_{k,\eps} -\kappa \psi_{k,\eps} ,
\\[1ex]
\dsp \overline \sl_{k,\eps}=   \sl_{k,\eps}     +\eps (2\bd(\psi_{k,\eps})\cdot \nn) +\frac \beta 2(\psi_{k,\eps}\cdot \nn)\nn,
\\[1ex]
\dsp \overline h_{k,\eps}  = h_{k,\eps}  +   \frac{1}{\eps} \psi_{k,\eps}\cdot \nn -\frac\beta 4  \psi_{k,\eps}\cdot \nn\ .
\end{array}$$

By assumption \eqref{regdata}, since $k=\bk-2$,  the terms $\sw_k^l$, $\sw_{k-1}^l$, $\nu_k^l$,  $\nu_{k-1}^l$, $\Vmb_k$, $\Vmb_{k-1}$ are bounded in $\H^2(\Omega_-)$, and $\Vplus_k$, $\Vplus_{k-1}$, $\Pplus_k$ and $\Pplus_{k-1}$ are bounded in $\H^2(\Omega_+)$.

Hence, we obtain the following estimates: there exists $C$ such that for all $\eps$,

\begin{equation}
\label{estifinales}
\begin{array}{l}
\Vert \overline{\sg}_{k,\eps}  \Vert_{\L^2(\Omega_-)}\leq C \eps^{\frac{k}{2}} \\
\\
\Vert \overline{\sl}_{k,\eps}  \Vert_{\L^2(\Sigma)}\leq C \eps^{\frac{k}{2}} \\
\\
\Vert \overline{h}_{k,\eps}  \Vert_{\L^2(\Sigma)}\leq C  \eps^{\frac{k-2}{2}} .
\end{array}
\end{equation}

Then, the remainder terms $(\overline \sr_{k,\eps}^-,\sr_{k,\eps}^+, \sq_{k,\eps}^-, \sq_{k,\eps}^+)$ satisfy Problem \eqref{0SBBJ} with the following data:
$$\sg^-=  \overline \sg_ {k,\eps}  ,\quad, \sg^+=0, \quad \sl= \overline \sl_{k,\eps}, \quad \sh=-\alpha \left( \psi_{k,\eps} \right)_\tau -\frac 1 2 \overline \sl_{k,\eps}-  \overline h_{k,\eps}\nn$$
and by \eqref{estifinales}, using Theorem \ref{thm32}, we obtain that

$$
 \eps \|\bd(\overline{\sr}_{k,\eps}^{-}) \|_{0,\Omega_{-}}^2
  +  \frac {\kappa}{4} \|\overline{\sr}_{k,\eps}^{-} \|^2_{0,\Omega_{-}}  
  + \frac{ \mu C^2}{3} \|{\sr}_{k,\eps}^{+} \|_{1,\Omega_{+}}^2
\leqslant
C\eps^{ \frac{k-2}{2}}.$$

Since $\sr_{k,\eps}^-=\overline{\sr}_{k,\eps}^- +\psi_{k,\eps}$,  this concludes the proof of Theorem \ref{DAS-WKB} by using estimate \eqref{Epsikeps}.
\end{proof}

Applying Theorem \ref{DAS-WKB} for $\bk = 5$, by straightforward estimates of the order 1, order 2 and order 3 terms of the asymptotic expansion, we conclude the proof of Theorem \ref{thm-conv}.

\vspace{5mm}

\bibliographystyle{plain}
\bibliography{biblio}

\end{document}